\documentclass[10pt]{amsart}

\usepackage{amssymb, dsfont} 
\usepackage{fullpage, graphicx} 
\usepackage{hyperref} 
\usepackage{enumitem} 
\usepackage{xcolor} 
\usepackage{thmtools} 
\usepackage{caption, subcaption} 

\hypersetup{colorlinks=true, urlcolor=purple, citecolor=blue, linkcolor=red, pdfstartview={XYZ null null 0.90}}

\DeclareMathOperator{\ID}{ID}
\DeclareMathOperator{\MC}{MC}
\DeclareMathOperator{\RMC}{RMC}

\DeclareMathOperator{\PSL}{PSL}
\DeclareMathOperator{\GL}{GL}
\DeclareMathOperator{\PGL}{PGL}
\DeclareMathOperator{\Id}{Id}

\newcommand{\ZZ}{\ensuremath{\mathbb{Z}}}
\newcommand{\RR}{\ensuremath{\mathbb{R}}}
\newcommand{\CC}{\ensuremath{\mathbb{C}}}
\newcommand{\Proj}{\ensuremath{\mathbb{P}}}
\newcommand{\uhp}{\ensuremath{\mathbb{H}}} 
\newcommand{\Ap}{\ensuremath{\mathbb{A}}}
\newcommand{\Apfull}{\ensuremath{\Ap_{\text{full}}}}
\newcommand{\BQFquad}{\ensuremath{\mathbb{B}}}
\newcommand{\fdom}{\ensuremath{U_{\PGL}}}

\newcommand{\lm}[1]{\ensuremath{\left(\begin{matrix} #1 \end{matrix}\right)}}
\newcommand{\sm}[1]{\ensuremath{\left(\begin{smallmatrix} #1 \end{smallmatrix}\right)}}
\newcommand{\gensm}{\sm{a & b \\ c & d}}
\newcommand{\genlm}{\lm{a & b \\ c & d}}
\newcommand{\spa}[1]{\ensuremath{\left\langle #1\right\rangle}}
\newcommand{\bol}[1]{\ensuremath{\boldsymbol{#1}}}

\newtheorem{theorem}{Theorem}[subsection]
\newtheorem{corollary}[theorem]{Corollary}
\newtheorem{lemma}[theorem]{Lemma}
\newtheorem{proposition}[theorem]{Proposition}

\theoremstyle{definition}
\newtheorem{definition}[theorem]{Definition}
\newtheorem{remark}[theorem]{Remark}

\numberwithin{equation}{subsection}

\begin{document}

\title{The Apollonian staircase}
\author[J. Rickards]{James Rickards}
\address{University of Colorado Boulder, Boulder, Colorado, USA}
\email{james.rickards@colorado.edu}
\urladdr{https://math.colorado.edu/~jari2770/}
\date{\today}
\thanks{I thank Elena Fuchs and Katherine E. Stange for many illuminating discussions, and the anonymous referee for their useful feedback. This work was partially supported by NSF-CAREER CNS-1652238 (PI Katherine E. Stange).}
\begin{abstract}
A circle of curvature $n\in\mathbb{Z}^+$ is a part of finitely many primitive integral Apollonian circle packings. Each such packing has a circle of minimal curvature $-c\leq 0$, and we study the distribution of $c/n$ across all primitive integral packings containing a circle of curvature $n$. As $n\rightarrow\infty$, the distribution is shown to tend towards a picture we name the Apollonian staircase. A consequence of the staircase is that if we choose a random circle packing containing a circle $C$ of curvature $n$, then the probability that $C$ is tangent to the outermost circle tends towards $3/\pi$. These results are found by using positive semidefinite quadratic forms to make $\mathbb{P}^1(\mathbb{C})$ a parameter space for (not necessarily integral) circle packings. Finally, we examine an aspect of the integral theory known as spikes. When $n$ is prime, the distribution of $c/n$ is extremely smooth, whereas when $n$ is composite, there are certain spikes that correspond to prime divisors of $n$ that are at most $\sqrt{n}$.
\end{abstract}
\maketitle

\setcounter{tocdepth}{1}
\tableofcontents

\section{Introduction}

A Descartes configuration is a set of four mutually tangent circles in the plane with disjoint interiors. We may add to this picture by choosing three of the circles, and drawing the other circle that is also mutually tangent to all three. By repeating this process, we get an Apollonian circle packing. If the four initial curvatures were all integral, then every curvature in the packing is integral, and we call this an integral Apollonian circle packing. See Figure \ref{fig:ACPintro} for an example of an integral packing, where the circles are labeled by curvature. Renewed interest in integral packings came with the work of Graham, Lagarias, Mallows, Wilks, and Yan in \cite{GLMWY02}, where many fundamental properties were documented.

\begin{figure}[t]
	\includegraphics{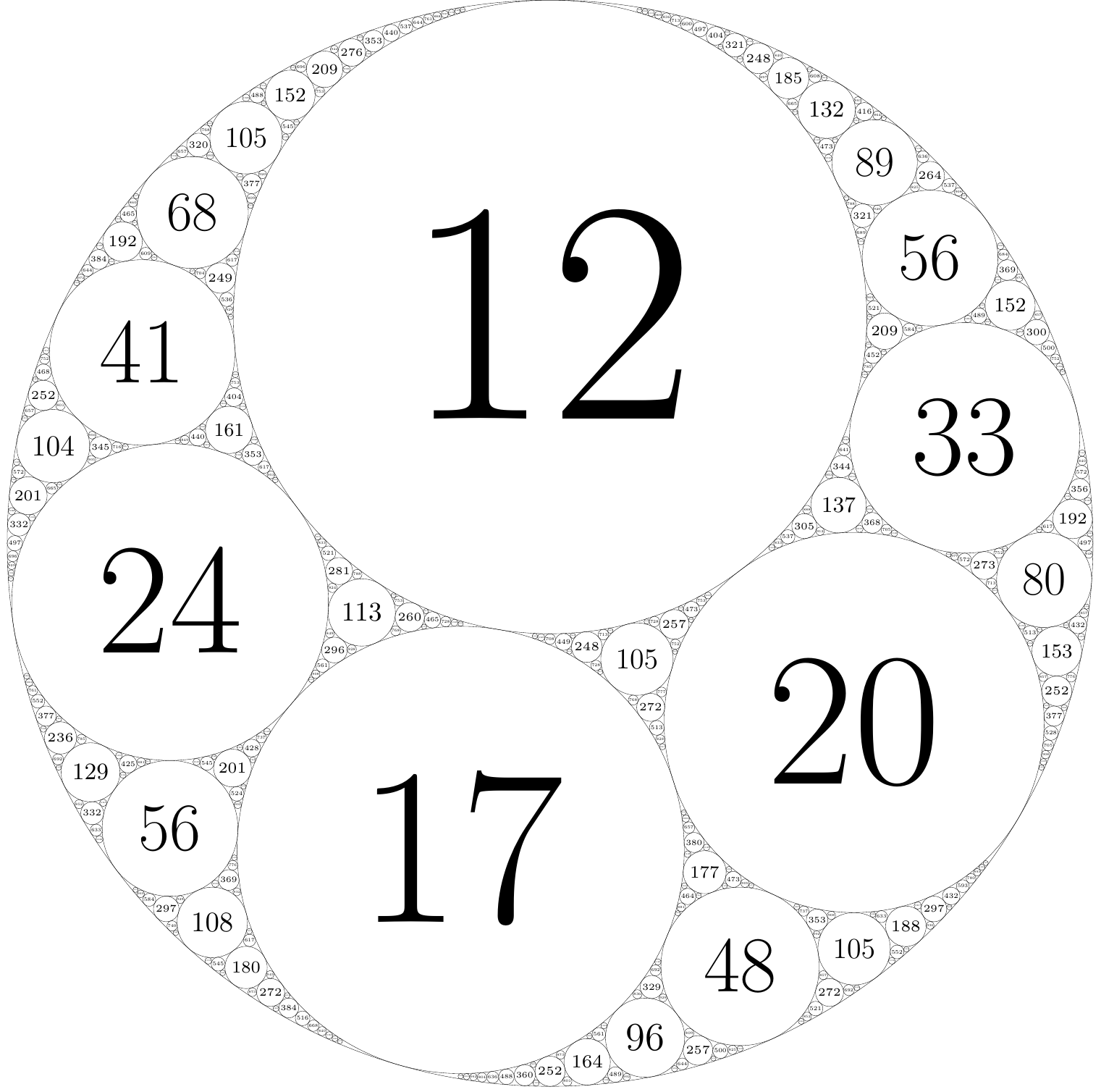}
	\caption{Apollonian circle packing corresponding to $(-7, 12, 17, 20)$.}\label{fig:ACPintro}
\end{figure}

Much of the recent work on Apollonian circle packings has centred around the asymptotic behaviour of the curvatures in an integral packing. One goal is to prove that all sufficiently large curvatures must appear in any given packing, up to congruence restrictions modulo 24. See \cite{BK14} and \cite{FSZ19} for partial results towards this conjecture. In this paper, we go in the other direction: start with a circle packing containing a circle of a given curvature, and consider how deep in the packing this circle lies. 

A related study was undertaken in the papers of Kocik (\cite{JK20}), and Holly (\cite{Holly21}). Both papers use $[0,1]^2$ as a parameter space for Apollonian circle packings (in slightly different ways), and show that the depths of circles in these packings creates an interesting fractal. In the paper of Holly, it is also shown that the location of the parameters in $[0,1]^2$ determines the nature of the corresponding packing, i.e. full plane, strip, half plane, or bounded. See their papers and Remark \ref{rem:connection} for more detail.

Another related paper is the work of Chaubey, Fuchs, Hines, and Stange in \cite{CFHS19}, where they find a continued fraction expansion for complex numbers using a Super-Apollonian packing. The idea is to walk through the circle packing via a sequence of tangency points, which is closely related to the idea of a depth element and depth circle, as studied in Section \ref{sec:quaddepth}.

A bounded packing has a unique circle of minimal (necessarily negative) curvature, which encloses all other circles. Similarly, half-plane and strip packings contain one and two (respectively) circles of curvature zero, and none of negative curvature. All integral packings are either bounded or strip.

\begin{definition}
Let $\bol{q}=(a, b, c, d)$ be a Descartes quadruple, i.e. four curvatures that correspond to a Descartes configuration, where a negative curvature indicates that the interior of the circle contains the point at infinity. If $\bol{q}$ does not generate a full plane packing, define $\MC(\bol{q})$ to be the \textit{negative} of the minimal curvature in the corresponding Apollonian circle packing. Otherwise, define $\MC(\bol{q})$ to be $0$.
\end{definition}

To study the asymptotic behaviour of $\MC$, fix a positive integer $n$, and consider the integral Descartes quadruples that contain $n$. Up to a reasonable definition of equivalence (Definition \ref{def:nquadequiv}), there are finitely many such quadruples, which are collected in the set $\ID(n)$ (``ID'' being ``integral Descartes'').

\begin{definition}
Define
\[\MC(n):=\{\MC(\bol{q}):\bol{q}\in\ID(n)\}\]
to be the multiset of negatives of minimal curvatures of quadruples containing $n$. Furthermore, define
\[\RMC(n):=\MC(n)/n=\{d/n:d\in\MC(n)\}\]
to be the ratios of curvatures in $\MC(n)$ to $n$ (also known as the ``heights'' of elements of $\ID(n)$).
\end{definition}

Since $\RMC(n)$ is contained in $[0, 1]$ and $|\ID(n)|\rightarrow\infty$ as $n\rightarrow\infty$, we can study the limiting distribution. It appears to converge to a distribution we call the ``Apollonian staircase''; see Figure \ref{fig:RMC1} for $\RMC(33920039)$ (all data in this paper was computed using PARI/GP \cite{PARI}).

\begin{figure}[bh]
	\includegraphics{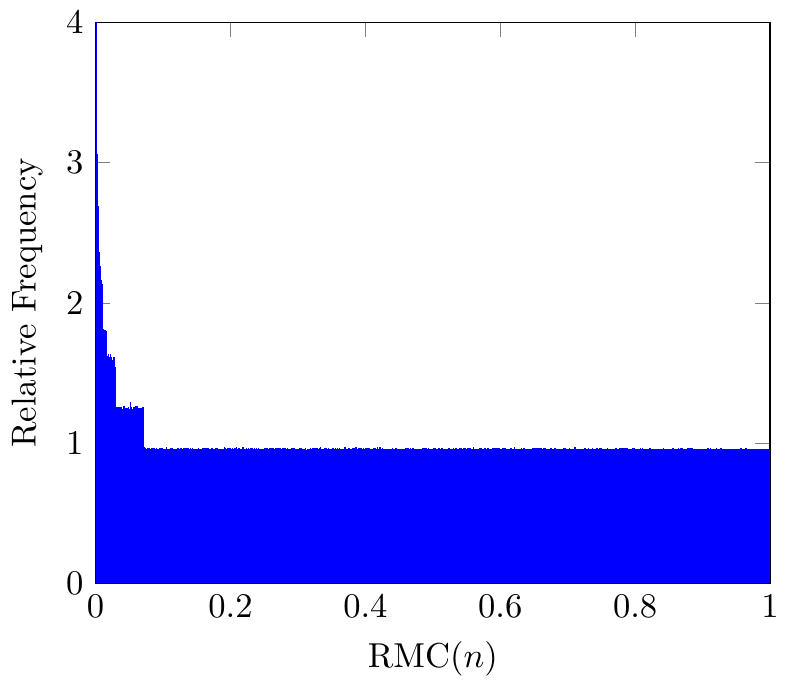}
	\caption{Histogram for $n=33920039$; $8480011$ data points in $2000$ bins.}\label{fig:RMC1}
\end{figure}

In particular, this appears to be piecewise uniform, with increasingly frequent jump discontinuities occurring near $0$. The different ``stairs'' correspond to different ``depths'' of the given circle in the corresponding circle packing. In Section \ref{sec:staircase} we precisely describe the Apollonian staircase, and prove the following theorem.

\begin{theorem}\label{thm:getstairs}
As $n\rightarrow\infty$, the distribution $\RMC(n)$ tends to the Apollonian staircase.
\end{theorem}

In order to prove this result, we give a direct connection between Descartes quadruples and $\PGL(2, \ZZ)$ equivalence classes of positive semidefinite binary quadratic forms, which was also considered in Theorem 4.2 of \cite{GLMWY02}. By considering where the principal root (Definition \ref{def:principalroot}) of the quadratic form lies in relation to the strip packing (embedded in $\CC$), we can describe the precise relationship between $\bol{q}$ and $\MC(\bol{q})$. An application of Duke's equidistribution theorem (\cite{Duke88}) allows us to specialize to primitive integral quadruples, and prove Theorem \ref{thm:getstairs}.

Another related phenomenon is the concept of ``spikes'' in the distribution, which is fully investigated in Section \ref{sec:spikes}. For example, take $n=42728555$, whose distribution is found in Figure \ref{fig:RMC2}.

\begin{figure}[hb]
	\includegraphics{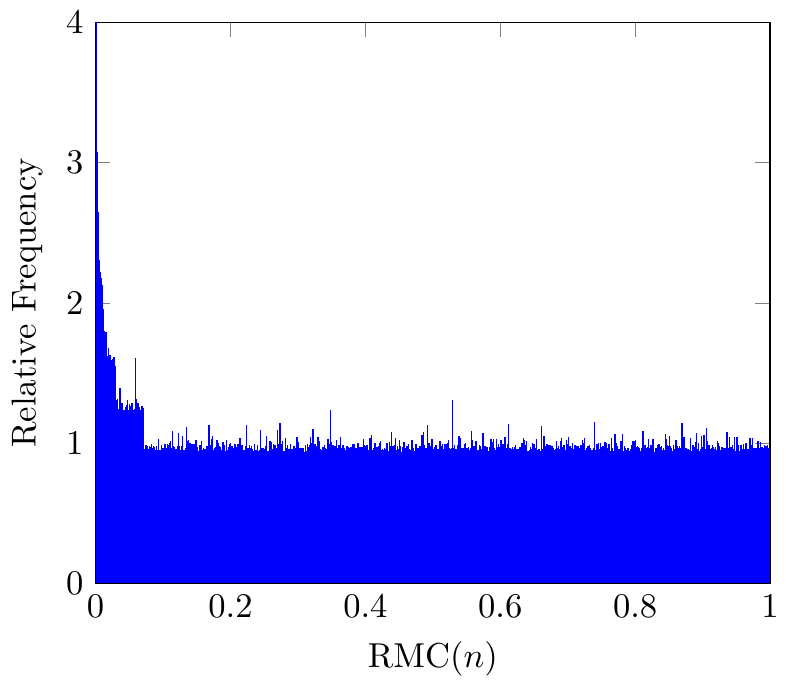}
	\caption{Histogram for $n=42728555$; $8480008$ data points in $2000$ bins.}\label{fig:RMC2}
\end{figure}

This is a lot rougher than Figure \ref{fig:RMC1}, despite similar amounts of data and bin sizes. The appearance of spikes is roughly described in the next theorem.

\begin{theorem}
Spikes appear in the histogram for $\RMC(n)$ for each prime $p\mid n$ with $p\leq\sqrt{n}$. Primes close to $\sqrt{n}$ give rise to a small number of tall spikes, whereas primes close to $1$ give rise to a large number of short spikes.
\end{theorem}

See Section \ref{sec:spikes} for a more precise description of how spikes occur. In terms of Figure \ref{fig:RMC2}, $n$ factorizes as
\[42728555=5\cdot 101\cdot 211\cdot 401,\]
all of which are primes at most $\sqrt{n}$, giving a wide variety of spikes. Note that the appearance of spikes does not affect Theorem \ref{thm:getstairs}, since that theorem concerns bins of fixed length as $n\rightarrow\infty$. The effect of the spikes is washed away as the cumulative frequency of each bin goes to infinity.

Finally, Theorem \ref{thm:getstairs} has an interesting numerical corollary. A circle is tangent to the outer circle in its corresponding packing if and only if it contributes to the bottom stair of the staircase. Using the description of the staircase, we can compute the probability that this situation occurs.

\begin{corollary}\label{cor:probtangent}
Pick a quadruple $\bol{q}$ uniformly at random from $\ID(n)$. Then as $n\rightarrow\infty$, the probability that the circle of curvature $n$ in $\bol{q}$ is tangent to the outermost circle in its corresponding Apollonian circle packing tends to $\frac{3}{\pi}$.
\end{corollary}

\begin{remark}
Curiously, the fraction $\frac{3}{\pi}$ also appears in the work of Athreya, Cobeli, and Zaharescu in \cite{ACZ15}. In their paper, they fix a circle $C$ in an Apollonian circle packing, and consider $\epsilon-$neighbourhoods of the exterior of $C$. It is shown that the proportion of points in the neighbourhood that lie in a circle tangent to $C$ tends to $\frac{3}{\pi}$ as $\epsilon\rightarrow 0$. Both questions deal with probabilities of circles being tangent in Apollonian circle packings, but the parameter spaces are quite different. It is not obvious if the appearance of $\frac{3}{\pi}$ in each place is an accident, or there is a deeper relation between the questions.
\end{remark}

Sections \ref{sec:apolgroup} and \ref{sec:psdbqf} precisely define the map $\bol{q}\rightarrow p_{\bol{q}}$, taking a Descartes quadruple $\bol{q}$ to a corresponding binary quadratic form, and finally to its principal root $p_{\bol{q}}\in\Proj^1(\CC)$. In Section \ref{sec:quaddepth}, the location of $p_{\bol{q}}$ with respect to an embedding of the strip packing is shown to determine the depth of $\bol{q}$. Section \ref{sec:quadheight} studies the heights of quadruples having $p_{\bol{q}}$ lying in a given part of the strip packing. In Section \ref{sec:fdomdist} we restrict $p_{\bol{q}}$ to be in the fundamental domain for $\PGL(2, \ZZ)$, give probabilities for the different depths of $\bol{q}$, and examine the distribution of heights. Finally, Section \ref{sec:spikes} considers integral Descartes quadruples, where we precisely describe the Apollonian staircase, and finish proving the main results of the introduction.

\section{The Apollonian group}\label{sec:apolgroup}

Given an (ordered) Descartes configuration, a ``move'' consists of replacing one of the four circles by the other circle that is tangent to the remaining three. There are four possible moves, denoted $S_1, S_2, S_3, S_4$, where $S_i$ corresponds to replacing the $i$\textsuperscript{th} circle.

\begin{definition}
Let $\Ap$ be the group generated by the $S_i$, called the Apollonian group. A reduced word in $\Ap$ is any sequence of the $S_i$ which does not contain the same element in consecutive positions.
\end{definition}

An element of $\Ap$ replaces a given Descartes configuration by another configuration in the corresponding Apollonian packing. If the ordering of the circles is ignored, this will generate all Descartes configurations in the packing.

Algebraically, assume we start with the Descartes quadruple $\bol{q}=(a, b, c, d)$, which satisfies the Descartes equation
\begin{equation}\label{eqn:descartes}
(a+b+c+d)^2=2(a^2+b^2+c^2+d^2).
\end{equation}
Vieta's formulas imply that the move $S_1$ replaces $a$ with $2(b+c+d)-a$. The group elements $S_i$ can be represented as $4\times 4$ matrices, acting on the column vectors $(a, b, c, d)^T$. For example,
\[S_1=\lm{-1 & 2 & 2 & 2\\0 & 1 & 0 & 0\\0 & 0 & 1 & 0\\0 & 0 & 0 & 1}\quad\text{and}\quad S_2=\lm{1 & 0 & 0 & 0\\2 & -1 & 2 & 2\\0 & 0 & 1 & 0\\0 & 0 & 0 & 1}.\]
This turns $\Ap$ into a subgroup of $\GL(4, \ZZ)$. Furthermore, it is a subgroup of the orthogonal group corresponding to the quadratic form 
\[Q_D:=\lm{1 & -1 & -1 & -1\\-1 & 1 & -1 & -1\\-1 & -1 & 1 & -1\\-1 & -1 & -1 & 1},\]
i.e. $W^TQ_DW=Q_D$ for all $W\in\Ap$. Each element of $\Ap$ can be written uniquely as a reduced word in $S_1, S_2, S_3, S_4$.

Since we are considering Descartes configurations/quadruples as being ordered, the orbit of a single configuration under $\Ap$ does not necessarily hit every configuration in the packing. To this end, if $\sigma$ is a permutation of $(1, 2, 3, 4)$, denote by $P_{\sigma}\in\GL(4, \ZZ)$ the corresponding action on a Descartes quadruple.

\begin{definition}
Define $\Apfull$ to be the group generated by the $S_i$ and the $P_{\sigma}$, which is still a subgroup of the orthogonal group corresponding to $Q_D$. Distinct orbits of $\Apfull$ correspond to distinct Apollonian circle packings.
\end{definition}

In order to talk about a specific circle in a packing, we take the first circle in a quadruple to be ``distinguished''.

\begin{definition}\label{def:nquadequiv}
Let $\Ap_1$ be the subgroup of $\Apfull$ generated by $P_{(23)}, P_{(24)}, S_4$. An $n-$quadruple refers to a Descartes quadruple of the form $(n, a, b, c)$. Two $n-$quadruples are declared equivalent if they are in the same $\Ap_1-$orbit.
\end{definition}

Note that any element of $\Ap_1$ can be written uniquely as $P_{\sigma}W$, where $\sigma$ is a permutation of $(1, 2, 3, 4)$ fixing $1$, and $W$ is a reduced word in $S_2$, $S_3$, $S_4$. In particular, quadruples in an $n-$quadruple class always start with the curvature $n$.

In most cases, an $n-$quadruple class will correspond to a unique circle in the geometric picture. However, in the strip packing, there are infinitely many circles that give rise to the same class. Similarly, in a packing coming from $(a, a, b, c)$, the two circles of curvature $a$ correspond to the same $a-$quadruple class. By working with $\Ap_1-$equivalence classes, we resolve the technical issues that arise from this.

Given a Descartes quadruple $\bol{q}$ corresponding to a bounded or half-plane packing, there is a unique reduced word $W\in\Ap$ such that $W\bol{q}$ contains a non-positive curvature. If $\bol{q}$ corresponds to the strip packing, there are two minimal words $W, W'$, one for each of the two curvature zero circles.

\begin{definition}
Define the depth of $\bol{q}$, $\delta(\bol{q})$, to be the length of $W$ if $\bol{q}$ is the bounded or half-plane packing, and the multiset of lengths of $W, W'$ for the strip packing. If $\bol{q}$ corresponds to a full plane packing, define $\delta(\bol{q})=\infty$. We say that $\bol{q}$ has depth $d$ if $d\in\delta(\bol{q})$. In particular, strip packing quadruples have one or two possible depths, and all other quadruples have a unique depth.
\end{definition}

The depth of a quadruple is a basic measure for how far away it is from containing the largest circle in a packing.

\begin{remark}\label{rem:connection}
This is essentially the same depth as defined by Kocik in \cite{JK20}. In this paper, he maps an Apollonian quadruple $(a, b, c, d)$ to $\left(\frac{a}{c}, \frac{b}{c}\right)\in [0,1]^2$, where it is assumed that $c=\max(a, b, c)$. Quadruples of a fixed depth correspond to unions of ellipses in $[0, 1]^2$, and this creates an interesting fractal. The analogous fractal is explored by Holly in \cite{Holly21}, where she maps $(a, b, c, d)$ to $\left(\frac{b}{c}, \frac{a}{b}\right)\in [0, 1]^2$, assuming that $a\leq b\leq c$. Points inside an ellipse correspond to bounded packings, the boundary of the ellipses minus tangency points are half-plane packings, tangency points of ellipses are strip packings, and any point not inside or on an ellipse gives a full-plane packing.
\end{remark}

By connecting Descartes quadruples to positive semidefinite quadratic forms, we generate a picture in $\Proj^1(\CC)$, which is analogous to the fractal from Holly and Kocik.

\section{Positive semidefinite binary quadratic forms}\label{sec:psdbqf}

\begin{definition}
Let $A, B, C\in\RR$ be not all zero, and consider the function $Q(x, y)=Ax^2+Bxy+Cy^2$, called a binary quadratic form. It can be written as $Q=[A, B, C]$, and has discriminant $D=B^2-4AC$. The form is definite (resp. semidefinite) if $D<0$ (resp. $D\leq 0$), and called positive if it only takes on nonnegative values for $x, y\in\RR$. Alternatively, a definite/semidefinite form is positive if and only if $A,C\geq 0$. Abbreviate positive definite binary quadratic form as PDBQF, and positive semidefinite as PSDBQF. For the rest of this paper, we will only be considering P(S)DBQFs.
\end{definition}

A real number $N$ is represented by $Q$ if there exist integers $x, y$ such that $Q(x, y)=N$. If there exist coprime integers $x, y$ with $Q(x, y)=N$, then we say $N$ is properly represented by $Q$.

The (right) action of $\PGL(2, \ZZ)$ on PSDBQF's is via
\[\gamma Q(x, y):=Q(ax+by, cx+dy),\text{ where }\gamma=\genlm.\]
This action preserves the discriminant, and divides the set of PSDBQF's into equivalence classes.

The classical theory of $\PGL(2, \ZZ)$ reduction of integral PDBQFs also applies to general PDBQFs. In particular, each equivalence class has a unique reduced representative, as defined in Definition \ref{def:reduced}.

\begin{definition}\label{def:reduced}
A PDBQF $[A, B, C]$ is called ($\PGL(2, \ZZ)-$)reduced if $0\leq B\leq A\leq C$.
\end{definition}

\subsection{Descartes quadruples and quadratic forms}

\begin{definition}
A \textit{BQF quadruple} is any quadruple $[n, A, B, C]\in\RR^4$ for which $[A, B, C]$ is a PSDBQF of discriminant $-4n^2$. It is called \textit{primitive integral} if $n,A,B,C\in\ZZ$ have no common factor. The action of $\gamma\in\PGL(2, \ZZ)$ on BQF quadruples is via
\[\gamma[n, A, B, C]:=[n, A', B', C'],\]
where $\gamma[A, B, C]=[A', B', C']$.
\end{definition}

The set of all BQF quadruples is thus given by 
\begin{equation}\label{eq:bdef}
\BQFquad:=\{[n, A, B, C]\neq \bol{0}:\quad A\geq 0,\quad C\geq 0,\quad 4n^2+B^2-4AC=0\}.
\end{equation}
We use square brackets and capital letters to distinguish BQF quadruples from Descartes quadruples. Theorem 4.2 of \cite{GLMWY02} furnishes the bijection between Descartes and BQF quadruples, and is recorded next (with updated notation).

\begin{proposition}\label{prop:quadqfbijection}
Let $n$ be a fixed real number. Then $n-$quadruples $(n, a, b, c)$ biject with BQF quadruples $[n, A, B, C]$ via the correspondence
\begin{align*}
\phi(n, a, b, c):= & [n, n+a, n+a+b-c, n+b],\\
\theta[n, A, B, C]:= & (n, A-n, C-n, A+C-B-n).
\end{align*}
Furthermore, primitive integral Descartes quadruples biject with primitive integral BQF quadruples.
\end{proposition}

Turning our focus to an individual circle makes the correspondence even stronger.

\begin{proposition}\label{prop:ap1orbit}
Let $n$ be a real number, and let $\bol{q}$ be an $n-$quadruple. Then the image of the $\Ap_1-$orbit of $\bol{q}$ under the map $\phi$ is the $\PGL(2, \ZZ)$ orbit of $\phi(\bol{q})$.
\end{proposition}
\begin{proof}
Let 
\[S=\lm{0 & 1\\-1 & 0}, \quad T=\lm{1 & 1\\0 & 1}, \quad U=\lm{1 & 0\\0 & -1},\]
which generate $\PGL(2, \ZZ)$. Let $\bol{Q}=[n, A, B, C]$ be a BQF quadruple, and a computation shows that
\begin{align*}
\theta(S\bol{Q}) = & S_4P_{(23)}\theta(\bol{Q});\\
\theta(T\bol{Q}) = & P_{(34)}S_4\theta(\bol{Q});\\
\theta(U\bol{Q}) = & S_4\theta(\bol{Q}).
\end{align*}
Thus the image of the $\PGL(2, \ZZ)$ orbit of $\bol{Q}$ corresponds to the orbit of $\theta(\bol{Q})$ under
\[\spa{S_4P_{(23)}, P_{(34)}S_4, S_4}=\Ap_1.\]
The result follows.
\end{proof}

A consequence of this result is that any circle touching the circle of curvature $n$ has a curvature that is properly represented $Q'-n$, where $\phi(\bol{q})=[n, Q']$. This property was first observed by Sarnak in \cite{Sar07}, and has been crucial in the aforementioned partial results towards the local-global conjecture for integral packings (\cite{BK14} and \cite{FSZ19}).

\begin{definition}
Let the matrix $S_{\theta}$ be defined by
\[S_{\theta}:=\lm{1 & 0 & 0 & 0\\-1 & 1 & 0 & 0\\-1 & 0 & 0 & 1\\-1 & 1 & -1 & 1},\]
so that
\[\theta[n, A, B, C]=\left(S_{\theta}[n, A, B, C]^T\right)^T.\]
\end{definition}

When using BQF quadruples as the parameter space, the action of the Apollonian group is via $S_{\theta}^{-1}\Ap S_{\theta}$. However, we need to consider curvatures, so we don't want to map back to BQF quadruples at the end. This amounts to working with the coset $\Ap S_{\theta}$ instead. Indeed, left multiplication of a BQF quadruple $\bol{Q}$ by $WS_{\theta}\in\Ap S_{\theta}$ corresponds to $W\theta(\bol{Q})$, i.e. the action of $W$ on the corresponding Descartes quadruple.

\begin{lemma}\label{lem:Apthetaqf}
Let $W_{\theta}\in\Ap S_{\theta}$. Then $W_{\theta}Q_{\theta}W_{\theta}^T=Q_D$, where
\[Q_{\theta}:=\lm{1 & 0 & 0 & 0\\0 & 0 & 0 & -2\\0 & 0 & 4 & 0\\0 & -2 & 0 & 0}.\]
\end{lemma}
\begin{proof}
Write $W_{\theta}=WS_{\theta}$, where $W\in\Ap$. A computation shows that $S_{\theta}Q_{\theta}S_{\theta}^T=Q_D$, whence the result follows if $WQ_DW^T=Q_D$. This is true for $S_i$ for $1\leq i\leq 4$, hence is true for all of $\Ap$.
\end{proof}

The equation $WQ_DW^T=Q_D$ holding for all $W\in\Ap$ is equivalent to $\Ap^T$ also being a subgroup of the orthogonal group corresponding to $Q_D$. A corollary of this lemma is that the rows of $W_{\theta}$ obey a quadratic relation.

\begin{corollary}\label{cor:tuvweqn}
Let $(t, u, v, w)$ be a row of $W_{\theta}\in\Ap S_{\theta}$. Then
\[t^2+4v^2-4uw=1.\]
\end{corollary}
\begin{proof}
In general, if the matrices $A,B,Q$ satisfy $A^TQA=B$, then $B_{ij}=\bol{a_i}^TQ\bol{a_j}$, where $\bol{a_i}$ is the $i$\textsuperscript{th} column of $A$. From Lemma \ref{lem:Apthetaqf}, this holds with $A=W_{\theta}^T$, $Q=Q_{\theta}$, and $B=Q_D$, and the result follows by taking $i=j$.
\end{proof}

\subsection{Principal root of a quadratic form}

We have transferred Descartes quadruples to BQF quadruples, and we now map the picture into $\Proj^1(\CC)$ by taking a root of the quadratic form.

\begin{definition}\label{def:principalroot}
Let $\bol{Q}=[n, A, B, C]$ be a BQF quadruple. The function $AX^2+BX+C$ has two roots (with multiplicity) in $\Proj^1(\CC)$; we designate one root as principal via the explicit definition
\[p_{\bol{Q}}:=\begin{cases}
\dfrac{-B+2ni}{2A} & \text{if $A\neq 0$;}\\
\infty & \text{if $A=0$.}
\end{cases}\]
\end{definition}

Note that $p_{\bol{Q}}$ is the upper half plane root of the corresponding quadratic form if $n>0$, and the lower half plane root if $n<0$. If $n=0$, there is a unique root in $\Proj^1(\CC)$.

\begin{definition}
Let $\gamma\in\PGL(2, \ZZ)$. The action of $\gamma=\gensm$ on $z\in\Proj^1(\CC)$ is defined as
\[\gamma z:=\begin{cases}
\frac{az+b}{cz+d} & \text{if $\det(\gamma)=1$;}\\
\frac{a\overline{z}+b}{c\overline{z}+d} & \text{if $\det(\gamma)=-1$.}
\end{cases}\]
\end{definition}

This action is via the corresponding M\"{o}bius map if $\det(\gamma)=1$, and the M\"{o}bius map acting on $\overline{z}$ otherwise. In particular, the upper half plane is preserved by the action of $\PGL(2, \ZZ)$.

\begin{proposition}
The action of $\PGL(2, \ZZ)$ on BQF quadruples $\bol{Q}$ commutes with the inverse action on $p_{\bol{Q}}$, i.e. $p_{\gamma\bol{Q}}=\gamma^{-1}(p_{\bol{Q}})$ for all $\gamma\in\PGL(2, \ZZ)$.
\end{proposition}
\begin{proof}
It suffices to check this claim on the generators $S, T, U$ of $\PGL(2, \ZZ)$ (from Proposition \ref{prop:ap1orbit}). If $z\in\Proj^1(\CC)$, then
\[S^{-1}z=\dfrac{-1}{z},\quad T^{-1}z=z-1, \quad U^{-1}z=-\overline{z}.\]
Write $\bol{Q}=[n, A, B, C]$, and then
\[S\bol{Q}=[n, C, -B, A],\quad T\bol{Q}=[n, A, B+2A, C+B+A], \quad U\bol{Q}=[n, A, -B, C].\]
The result follows by direct computation.
\end{proof}

If $n>0$, then a reduced BQF quadruple $\bol{Q}=[n, A, B, C]$ corresponds to $p_{\bol{Q}}$ living in the fundamental region, as seen in Figure \ref{fig:gl2zdom}. Note that this is half of the classical fundamental domain for $\PSL(2, \ZZ)$, with the difference due to $\GL$: the action of $U$ folds the right half of the classical fundamental domain onto the left.

\begin{figure}[htb]
	\includegraphics{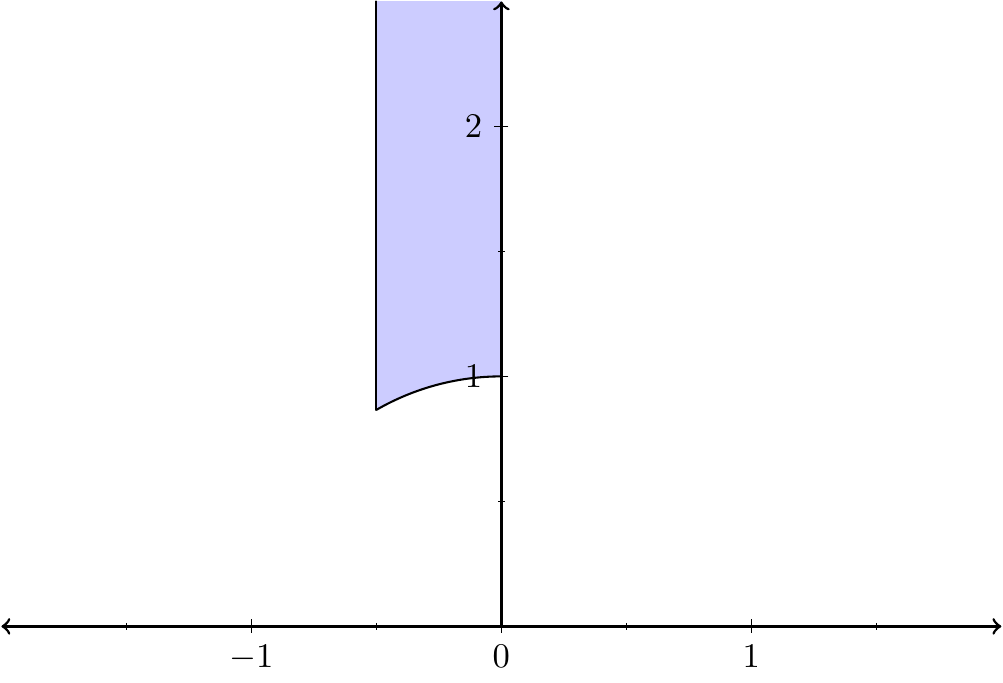}
	\caption{Fundamental domain for $\PGL(2, \ZZ)$.}\label{fig:gl2zdom}
\end{figure}

Recall the set of all BQF quadruples, $\BQFquad$, as defined in Equation \eqref{eq:bdef}.

\begin{lemma}\label{lem:qftorootbijection}
The map from $\BQFquad/\RR^+$ to $\Proj^1(\CC)$ via $\bol{Q}\rightarrow p_{\bol{Q}}$ is a bijection.
\end{lemma}
\begin{proof}
If $A=0$, then $0\leq B^2=-4n^2\leq 0$, whence $B=n=0$. Thus
\[[n, A, B, C]=[0, 0, 0, C]\sim[0, 0, 0, 1],\]
so there is a unique element that maps to $\infty$.

Otherwise, $A\neq 0$, write $p_{\bol{Q}}=x+iy$ for unique real numbers $x, y$, and we have:
\begin{equation}\label{eqn:xyexpressions}
x=\dfrac{-B}{2A},\quad y=\dfrac{n}{A},\quad x^2+y^2=\dfrac{C}{A}.
\end{equation}
There is a unique scaling of $[n, A, B, C]$ so that $A=1$, where $x=-B/2$, $y=n$. In particular, if $[n, 1, B, C]$ and $[n', 1, B', C']$ are mapped to the same point, then $B=B'$ and $n=n'$. Since $C=x^2+y^2=C'$ as well, the map is one to one. Finally, given $x,y\in\RR$, we obtain $B=-2x$, $n=y$, $C=x^2+y^2$, which implies that the map is onto, and thus a bijection.
\end{proof}

Combining Proposition \ref{prop:quadqfbijection} with Lemma \ref{lem:qftorootbijection} gives the following lemma.

\begin{lemma}\label{lem:apolquadprojbiject}
There is a bijection between Descartes quadruples up to scaling by $\RR^+$ and $\Proj^1(\CC)$, with the association being
\[\bol{q}\rightarrow p_{\phi(\bol{q})}:=p_{\bol{q}}.\]
\end{lemma} 

Proposition \ref{prop:ap1orbit} implies that an $n-$quadruple class containing $\bol{q}$ is taken to the $\PGL(2, \ZZ)$ orbit of $p_{\bol{q}}$ . In particular, for $n>0$, there is a unique representative of the class in the fundamental domain. For more on this, see Section \ref{sec:fdomdist}.

\section{Quadruple depth}\label{sec:quaddepth}
Since the depth of a quadruple, $\delta(\bol{q})$, is constant upon scaling by $\RR^+$, we can study the depth of Descartes quadruples by transferring the picture to $\Proj^1(\CC)$. If $\bol{q}$ generates a bounded or half-plane packing, there is a unique shortest reduced word $W$ such that $W\bol{q}$ contains a non-positive curvature. If $W\neq \Id$ and $W$ starts with $S_j$, then this circle must appear in the $j$\textsuperscript{th} position in $W\bol{q}$. Otherwise, it appears in $\bol{q}$, and can be in any position. This motivates the following definition.

\begin{definition}
A depth element $W$ is either a non-identity reduced word in $\Ap$, or an integer between $1$ and $4$. In the latter case, we write $W=\Id_j$ to refer to the depth element corresponding to the integer $j$.
\end{definition}

In particular, the above construction associates a unique depth element to $\bol{q}$. If $\bol{q}$ generates the strip packing, then the same result holds, except there are two circles of minimal curvature, and we get two depth elements.

\begin{definition}
If $W$ is a depth element, define
\[D_W:=\{p_{\bol{q}}:W\text{ is a depth element for $\bol{q}$}\}\subseteq\Proj^1(\CC),\]
which we refer to as a depth circle. For $m\geq 0$ an integer, define
\[D_m:=\{p_{\bol{q}}:m\in\delta(\bol{q})\}\subseteq\Proj^1(\CC).\]
Note that $D_m$ is the union of $D_W$ over all $W$ with length $m$.
\end{definition}

It is clear that $D_W$ (and therefore $D_m$) is a closed set. The boundary of $D_W$ consists of the set of $p_{\bol{q}}\in D_W$ for which $\MC(\bol{q})=0$ (a half-plane or strip packing), and the interior consists of the set of $p_{\bol{q}}\in D_W$ for which $\MC(\bol{q})>0$ (a bounded packing).

Furthermore, if $W\neq W'$, then the interiors of $D_W$ and $D_{W'}$ are disjoint. Any point in their intersection corresponds to a packing containing two circles of curvature zero, which is necessarily the strip packing (scaled). How do the regions $D_W$ subdivide $\Proj^1(\CC)$? Start with $D_0$, which is the union of $D_{\Id_j}$ for $1\leq j\leq 4$. Write $\phi(\bol{q})=[n, A, B, C]$, and this respectively corresponds to the four inequalities:
\begin{equation}\label{eqn:first4circles}
n\leq 0,\quad A-n\leq 0,\quad C-n\leq 0,\quad A+C-B-n\leq 0.
\end{equation}
If $A=0$ then $B=n=0$ and $C>0$, so only the first two inequalities are true. Otherwise, dividing by $A>0$ and using the expressions for $x, y$ in Equation \eqref{eqn:xyexpressions}, these inequalities respectively give
\begin{equation}\label{eqn:depth0}
y\leq 0, \quad y\geq 1, \quad x^2+(y-1/2)^2\leq 1/4, \quad (x+1)^2+(y-1/2)^2\leq 1/4.
\end{equation}
Thus $D_{\Id_j}$ is a circle for all $1\leq j\leq 4$ (with the convention that a half-plane is a circle with infinite radius), and the picture is depicted in Figure \ref{fig:depth0}. Observe that the four circles form a Descartes configuration, a part of the Apollonian strip packing scaled by $\frac{1}{2}$ and positioned between the $x-$axis and $y=1$.

\begin{figure}[htb]
	\includegraphics[scale=1.7]{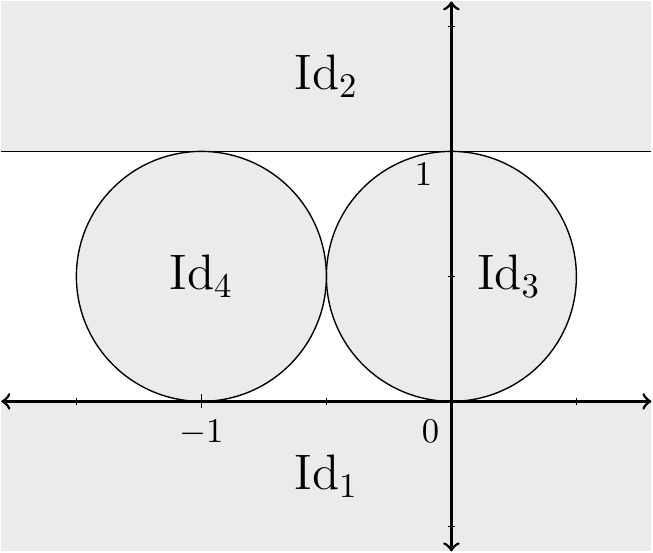}
	\caption{Depth circles with $\delta(\bol{q})=0$, labeled by depth element.}\label{fig:depth0}
\end{figure}

Going further, consider a general $D_W$ that starts with $S_j$. This region is determined by the equation
\[(W\bol{q})_j\leq 0,\]
which takes the form $-tn+uA+vB+wC\leq 0$, for the integers $t, u, v, w$ defined by
\[\text{Row $j$ of }WS_{\theta}=(-t, u, v, w).\]
Assume $A>0$ ($A=0$ corresponds to $\infty$, which can be added back in later), and dividing by $A$ yields
\[-ty+u-2vx+w(x^2+y^2)\leq 0.\]
If $w=0$, this gives a half-plane, whose interior must intersect with the interior of either $D_{\Id_1}$ or $D_{\Id_2}$, a contradiction. Therefore $w\neq 0$, divide by $w$, and rearrange to get
\begin{equation}\label{eqn:DPcircle}
\left(x-\dfrac{v}{w}\right)^2+\left(y-\dfrac{t}{2w}\right)^2\leq \dfrac{t^2+4v^2-4uw}{4w^2}=\dfrac{1}{4w^2},
\end{equation}
where the last equality is due to Corollary \ref{cor:tuvweqn}. Since we divided by $w$, we must switch the inequality if $w<0$. However, this would correspond to the exterior of a circle, which is not possible since the interior would again intersect the interior of the half-planes in $D_0$. Therefore $w>0$, and we obtain a circle and its interior as the solution set. Note that this also implies that $t>0$, as the centre needs to be in the upper half plane.

This discussion has proven the following lemma.

\begin{lemma}
Let $W$ be a depth element corresponding to the equation $-tn+uA+vB+wC\leq 0$. Then $D_W$ is a circle, defined by Equation \eqref{eqn:depth0} if $W=\Id_j$, and Equation \eqref{eqn:DPcircle} otherwise.
\end{lemma}

\begin{definition}
The coefficient quadruple corresponding to the depth element $W$ is the integral quadruple $(t, u, v, w)$. As long as $W\neq\Id_j$, we have $t, u, w>0$.
\end{definition}

To see how these circles fit together, define $D:=\cup_{m=0}^{\infty} D_m$, and consider Figure \ref{fig:depth2}, which depicts $D_0\cup D_1\cup D_2$.

\begin{figure}[htb]
	\includegraphics[scale=0.98]{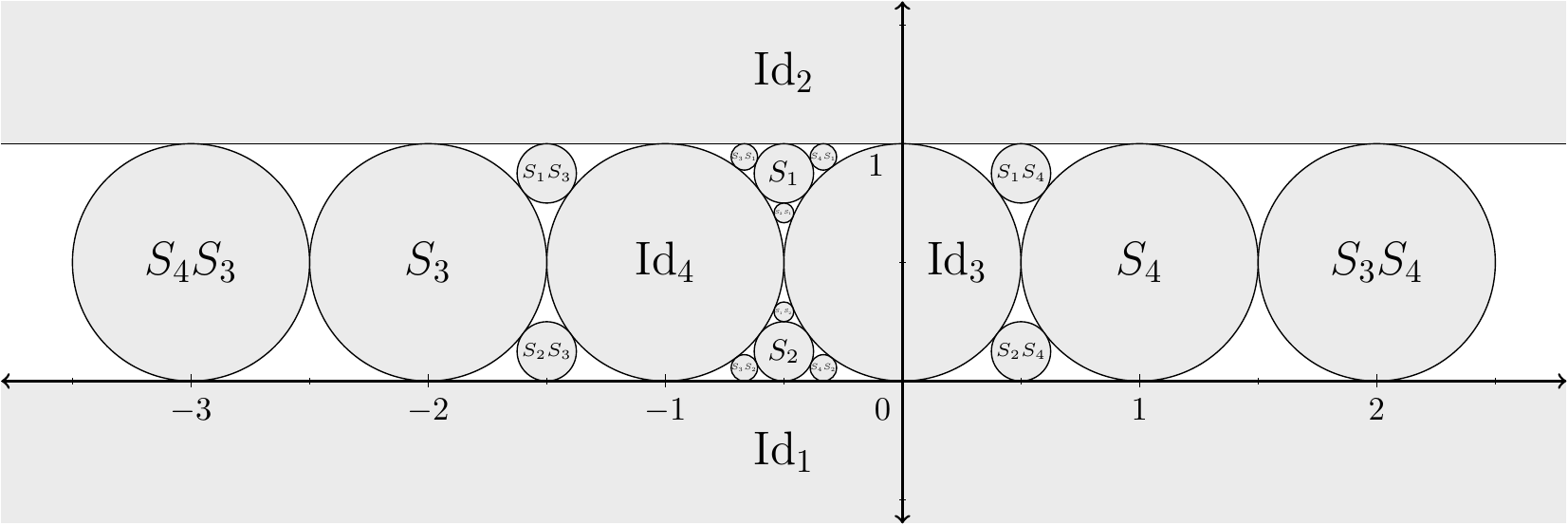}
	\caption{Depth circles with $\delta(\bol{q})\leq 2$, labelled by depth element.}\label{fig:depth2}
\end{figure}

We appear to be continuing the strip packing!

\begin{lemma}\label{lem:fourtangent}
Let $W\in\Ap$, and consider the four circles corresponding to $(W\bol{q})_i\leq 0$, $1\leq i\leq 4$. Then these circles are mutually tangent.
\end{lemma}
\begin{proof}
It suffices to prove that for each pair $(i,j)$ with $1\leq i<j\leq 4$, there is a Descartes quadruple $\bol{q}$ such that $W\bol{q}$ has zeroes in the $i$\textsuperscript{th} and $j$\textsuperscript{th} positions. Indeed, this would imply that the circles corresponding to $(W\bol{q})_i\leq 0$ and $(W\bol{q})_j\leq 0$ share the point $p_{\bol{q}}$, whence they intersect. They must be tangent as otherwise, their interiors would overlap, a contradiction.

To prove this claim, let $\bol{q}'$ be the permutation of $(0,0,1,1)$ having zeroes in positions $i$ and $j$, and take $\bol{q}=W^{-1}\bol{q}'$.
\end{proof}

\begin{theorem}
The set $D$, the union of all depth circles, is the strip packing scaled by $\frac{1}{2}$.
\end{theorem}
\begin{proof}
As seen in Figure \ref{fig:depth0}, $D_0$ is the start of the strip packing scaled by $\frac{1}{2}$. Take $W$ to be a depth element of length $m>0$ that begins with $S_j$. By Lemma \ref{lem:fourtangent}, the circles corresponding to $(W\bol{q})_i\leq 0$ are mutually tangent for $1\leq i\leq 4$. If $i=j$, this is $D_W$, and if $i\neq j$, this is a circle in $D_{m'}$ for some $m'<m$. Therefore we are drawing the fourth circle in a Descartes configuration, where three of the circles are present in $\cup_{k=0}^{m-1} D_k$. Thus, adding in the circles in $D_m$ corresponds to going one level deeper in the strip packing, and we generate the entire strip packing as $m\rightarrow\infty$.
\end{proof}

\begin{remark}
The strip packing is the analogue of the fractals of Kocik and Holly (see Remark \ref{rem:connection}). In particular, if $\bol{q}$ is a Descartes quadruple, then $\bol{q}$ generates a
\begin{itemize}
\item bounded packing if and only if $p_{\bol{q}}$ lies in the interior of a depth circle;
\item half-plane packing if and only if $p_{\bol{q}}$ lies on the boundary of a unique depth circle;
\item strip packing if and only if $p_{\bol{q}}$ is the tangency point of two depth circles;
\item full-plane packing if and only if $p_{\bol{q}}$ is not contained in any depth circle.
\end{itemize}
\end{remark}

\section{Quadruple height}\label{sec:quadheight}

\begin{definition}
Let $\bol{q}$ be an $n-$quadruple with $n>0$. Define the height of $\bol{q}$ to be
\[H(\bol{q}):=\frac{\MC(\bol{q})}{n}\in [0, 1).\]
\end{definition}

If $\delta(\bol{q})=0$, there are no obvious biases for where $H(\bol{q})$ should lie in $[0, 1)$. On the other hand, if $\delta(\bol{q})>0$, then there may be a layer of circles between $\bol{q}$ and the circle of smallest curvature, whence $H(\bol{q})$ would be somewhat small. To this end, we study the behaviour of $H(\bol{q})$ on the sets $D_W$ for $W\neq\Id_1$ (which corresponds to $n\leq 0$).

\begin{proposition}\label{prop:Hbound}
Let $W\neq\Id_1$ be a depth element with coefficient quadruple $(t, u, v, w)$. Then
\[0\leq H(\bol{q})\leq t-\sqrt{t^2-1},\]
whenever $p_{\bol{q}}\in D_W$.
\end{proposition}
\begin{proof}
If $W=\Id_2$, $t=1$ and $D_W$ is given by $y\geq 1$. We compute
\[H(\bol{q})=1-\dfrac{A}{n}=1-\dfrac{1}{y},\]
so all heights in $[0, 1-\sqrt{1^2-1})$ are possible.

Otherwise, $D_W$ is a circle, and $t,w>0$. Since $p_{\bol{q}}\in D_W$, it follows that $H(\bol{q})=-tn+uA+vB+wC$, where $\phi(\bol{q})=[n, A, B, C]$. Thus
\begin{equation}\label{eqn:Hofq}
H(\bol{q})=-t+u\dfrac{A}{n}+v\dfrac{B}{n}+w\dfrac{C}{n}=\dfrac{w}{y}\left(\dfrac{1}{4w^2}-\left(x-\frac{v}{w}\right)^2-\left(y-\frac{t}{2w}\right)^2\right),
\end{equation}
where $p_{\bol{q}}=x+iy$. This is continuous with respect to $x$ and $y$, and clearly hits the minimum of $0$ on the boundary of $D_W$. The maximal value must have $x=\frac{v}{w}$, whence we maximize the function
\[f(y)=\dfrac{1}{4wy}-\dfrac{w}{y}\left(y-\frac{t}{2w}\right)^2.\]
Taking the derivative and setting it to zero yields $y=\pm\frac{\sqrt{t^2-1}}{2w}$, and the positive root is a local maximum. Since
\[\dfrac{t-1}{2w}\leq\dfrac{\sqrt{t^2-1}}{2w}<\dfrac{t+1}{2w},\]
the local maximum falls inside $D_W$, and therefore furnishes the maximum value on $D_W$. Plugging in this value into the equation for $H(\bol{q})$ gives the result.
\end{proof}

A follow-up to Proposition \ref{prop:Hbound} is to consider the distribution of $H(\bol{q})$ with respect to the hyperbolic metric, when $D_W$ does not touch the real line. First, an expression for the hyperbolic area of a Euclidean circle is required.

\begin{lemma}\label{lem:circleharea}
Let $C$ be a circle in $\uhp$ with centre $a+hi$ and radius $r$, with $h>r$. Then the hyperbolic area of $C$ is
\[2\pi\left(\dfrac{h}{\sqrt{h^2-r^2}}-1\right).\]
\end{lemma}
\begin{proof}
This is classical; see, for example, Lemma 2.2. of \cite{AS82}.
\end{proof}

\begin{proposition}\label{prop:heightuniform}
Let $W$ be a depth element, where $D_W$ is a circle that does not touch the real axis. Then the values of $H(\bol{q})$ for $p_{\bol{q}}\in D_W$ are uniform in $[0, t-\sqrt{t^2-1}]$ with respect to the hyperbolic metric on $D_W$.
\end{proposition}
\begin{proof}
It suffices to compute the hyperbolic area of the set of points $p_{\bol{q}}\in D_W$ with $H(\bol{q})\geq t-\sqrt{t^2-1}-\epsilon$, and show that it grows linearly with $\epsilon$. To this end, using the expression for $H(\bol{q})$ in Equation \eqref{eqn:Hofq}, this inequality is true if and only if
\begin{equation}\label{eqn:Cepsilon}
\left(x-\frac{v}{w}\right)^2+\left(y-\frac{\sqrt{t^2-1}+\epsilon}{2w}\right)^2\leq \dfrac{\epsilon^2+2\epsilon\sqrt{t^2-1}}{4w^2}.
\end{equation}
For $\epsilon=0$, this is a circle of radius $0$ centred at $\left(\frac{v}{w}, \frac{\sqrt{t^2-1}}{2w}\right)$. As $\epsilon$ grows towards $t-\sqrt{t^2-1}$, the centre of the circle moves up and the radius increases, until finally we hit $D_W$. In particular, the region formed is a circle that is always contained inside $D_W$.

By Lemma \ref{lem:circleharea}, the hyperbolic area of the circle is
\[\dfrac{2\pi}{\sqrt{t^2-1}}\epsilon,\]
which proves the claim.
\end{proof}

Another consequence of Proposition \ref{prop:heightuniform} is that $t>1$ if and only if the circle does not touch the real axis. This could alternatively be demonstrated by showing that $S_2, S_3, S_4$ all fix the vector $(1, -1, -1, -1)^T$.

\begin{definition}
For $\epsilon\in[0, t-\sqrt{t^2-1}]$, define the $\epsilon-$circle of $W$ to be $C_W^{\epsilon}$, which is defined by Equation \eqref{eqn:Cepsilon}.
\end{definition}

Figure \ref{fig:epcirc} demonstrates a few $\epsilon-$circles for $D_{S_1}$, which has coefficient quadruple $(7, 4, -2, 4)$.

\begin{figure}[htb]
	\includegraphics{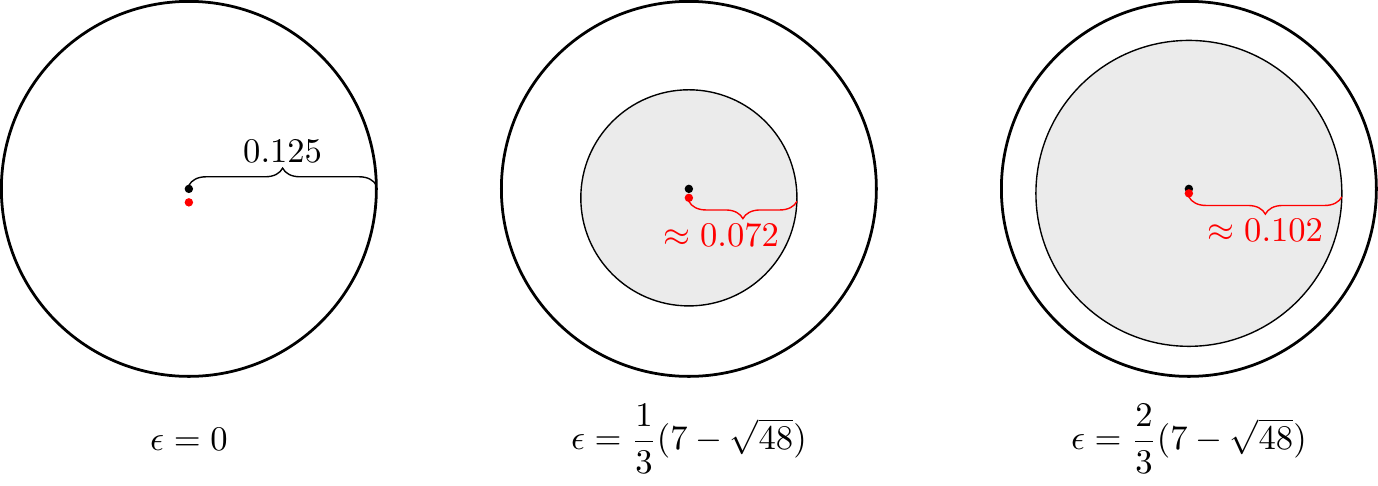}
	\caption{$C_{S_1}^{\epsilon}$ for three values of $\epsilon$. The red dots indicate the centres of $C_{S_1}^{\epsilon}$, while the black dots indicate the centres of $D_{S_1}$.}\label{fig:epcirc}
\end{figure}

\section{Fundamental domain distribution}\label{sec:fdomdist}

Take $\fdom$ to be the fundamental domain for $\PGL(2, \ZZ)$ as given in Figure \ref{fig:gl2zdom}, i.e. bounded by $x=-\frac{1}{2}$, $x=0$, and $x^2+y^2=1$. Let $\mu(\cdot)$ denote the hyperbolic area of a region of $\uhp$; it is well known that $\mu(\fdom)=\frac{\pi}{6}$.

If $n>0$, Lemma \ref{lem:apolquadprojbiject} implies that an $n-$quadruple class corresponds to a $\PGL(2, \ZZ)$ orbit of a point in $\uhp$. In particular, it corresponds to a unique point in $\fdom$. Thus we can produce a random $n-$quadruple class by picking a point uniformly at random in $\fdom$ with respect to the hyperbolic metric. 

\begin{definition}
Let $n>0$ and $p\in\uhp$. Denote the $n-$quadruple corresponding to $p$ by $\bol{\alpha}(p)$.
\end{definition}

The notation $\bol{\alpha}(p)$ depends on $n$, but $n$ will always be fixed, so no confusion will arise. The goal of this section is to prove the following theorem.

\begin{theorem}\label{thm:depthprob}
Let $W$ be a depth element with coefficient quadruple $(t, u, v, w)$, and choose a point $p\in\fdom$ uniformly at random with respect to the hyperbolic metric. Then the probability that $\bol{\alpha}(p)$ has depth element $W$ is given by
\[d_W:=\begin{cases}
0 & \text{if $\mu(D_W\bigcap \fdom)=0$;}\\
\dfrac{3}{\pi} & \text{ if $W=\Id_2$;}\\
a_W\left(\dfrac{t}{\sqrt{t^2-1}}-1\right) & \text{otherwise,}
\end{cases}\]
where
\[a_W:=\begin{cases}
2 & \text{if $W=S_1$;}\\
6 & \text{if $W=(S_4S_1)^k$ or $W=S_1(S_4S_1)^k$ with $k\geq 1$;}\\
12 & \text{otherwise.}
\end{cases}\]
Furthermore, $H(\bol{\alpha}(p))$ is distributed uniformly in $[0, t-\sqrt{t^2-1}]$ for $p\in D_W\bigcap\fdom$.
\end{theorem}

Consider Figure \ref{fig:fdomD4}, which depicts the intersections of $D_4$ with $\fdom$.

\begin{figure}[htb]
	\includegraphics[scale=1.1]{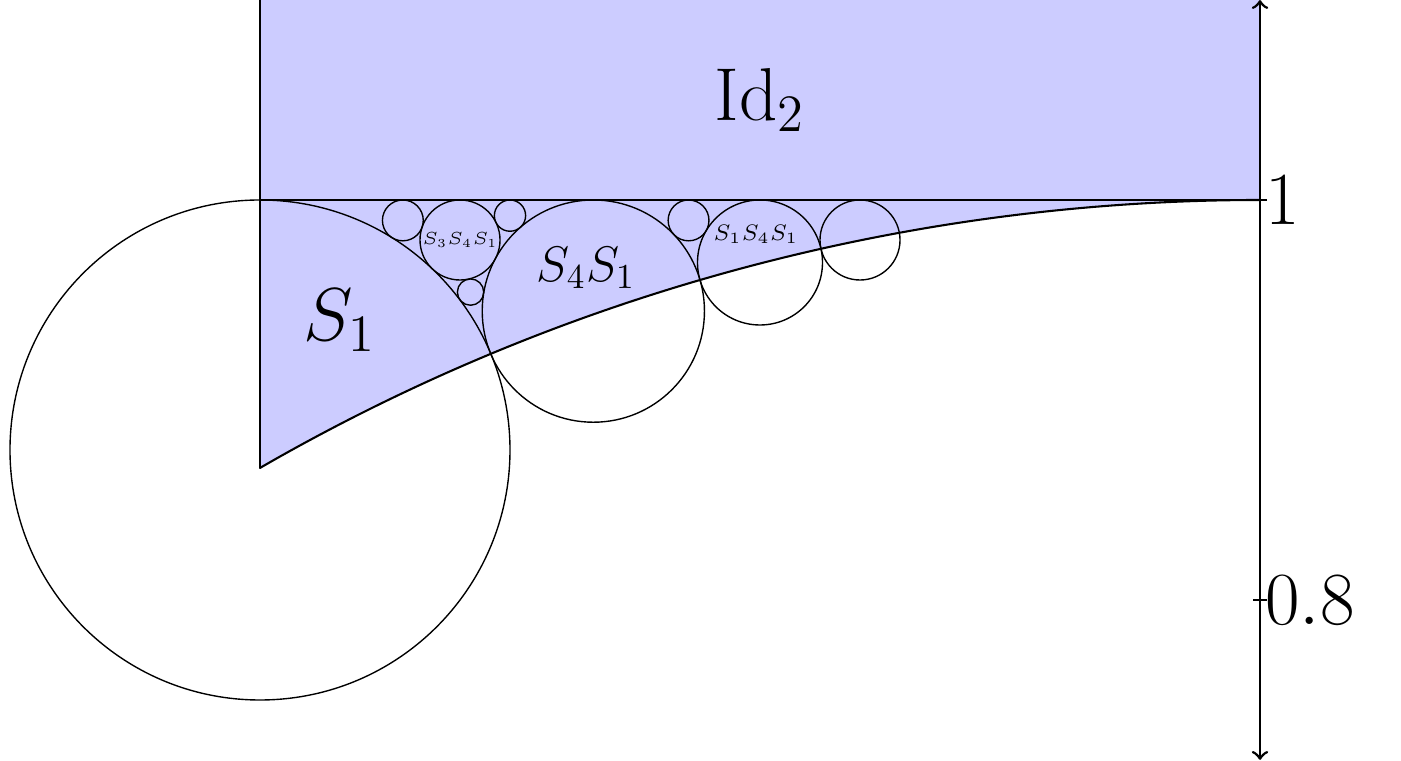}
	\caption{$D_4\bigcap\fdom$.}\label{fig:fdomD4}
\end{figure}

The structure of depth elements that intersect $\fdom$ is clear from Figure \ref{fig:fdomD4}. Define the sequence of depth elements $W_i$ by
\[W_0=\Id_2,\quad W_1=S_1,\quad W_2=S_4S_1,\quad W_3=S_1S_4S_1, \ldots,\]
so that $W_i$ is formed by alternately multiplying $\Id$ on the left by $S_1$ and then $S_4$. Then,
\begin{itemize}
\item $W_0$ cuts off the top part of $\fdom$;
\item $W_1$ cuts off the rest of the left side of $\fdom$;
\item $W_k$ for $k\geq 2$ cut off the rest of the bottom of $\fdom$;
\item All other $D_W$ that intersect $\fdom$ take the form $W=W'W_k$, where $k\geq 2$ and $W'\in\Ap$ is a reduced word ending in $S_3$. All such $W$ have $D_W$ lying entirely within $\fdom$.
\end{itemize}

The only claim that requires extra numerical justification is showing that the intersection point of $D_{W_k}$ with $D_{W_{k+1}}$ is on the unit circle for $k\geq 1$ (so that these circles carve out the bottom of $\fdom$).

\begin{lemma}\label{lem:wkcentre}
For $k\geq 1$, the intersection point of $D_{W_k}$ with $D_{W_{k+1}}$ lies on the unit circle.
\end{lemma}
\begin{proof}
Let the coefficient quadruple of $W_k$ be $(t_k, u_k, v_k, w_k)$. If $k$ is odd, it follows that
\[W_kS_{\theta}=\lm{-t_k & u_k & v_k & w_k\\-1 & 1 & 0 & 0\\-1 & 0 & 0 & 1\\-t_{k-1} & u_{k-1} & w_{k-1} & v_{k-1}}.\]
If $k$ is even, then the top row has the indices $k-1$, and the bottom row has indices $k$.

We claim that
\[(t_k, u_k, v_k, w_k)=(2(k+1)^2-1, (k+1)^2, -(k+1), (k+1)^2).\]
This is true for $k=1$, and follows by induction from multiplying the expression for $W_kS_{\theta}$ on the left by either $S_1$ or $S_4$. In particular, by Equation \eqref{eqn:DPcircle},
\[\text{the circle $D_{W_k}$ has centre}\left(-\dfrac{1}{k+1}, \dfrac{2(k+1)^2-1}{2(k+1)^2}\right)\text{ and radius }\dfrac{1}{2(k+1)^2}.\]

The intersection point of $D_{W_k}$ with $D_{W_{k+1}}$ can be computed to be
\[\left(-\dfrac{2k+3}{2k^2+6k+5}, \dfrac{2k^2+6k+4}{2k^2+6k+5}\right),\]
which lies on the unit circle.
\end{proof}

We prove Theorem \ref{thm:depthprob} by considering the various cases, as described above the previous lemma.

\begin{lemma}\label{lem:Wk3}
Let $W=W'W_k$, where $k\geq 2$ and $W'\in\Ap$ is a reduced word ending in $S_3$. Then Theorem \ref{thm:depthprob} holds for $W$.
\end{lemma}
\begin{proof}
Since $D_W$ lies entirely inside $\fdom$, the probability that a uniformly chosen point lands inside it is $\mu(D_W)/\mu(\fdom)$. Combining Equation \eqref{eqn:DPcircle} and Lemma \ref{lem:circleharea}, we compute
\[\mu(D_W)=2\pi\left(\dfrac{t/2w}{\sqrt{(t/2w)^2-(1/2w)^2}}-1\right)=2\pi\left(\dfrac{t}{\sqrt{t^2-1}}-1\right).\]
Dividing by $\mu(\fdom)=\pi/6$ gives the claimed probability. The distribution of $H(\bol{\alpha}(p))$ follows from Proposition \ref{prop:heightuniform}.
\end{proof}

The other easy case is for $W=W_0$.

\begin{lemma}\label{lem:Wk0}
Theorem \ref{thm:depthprob} holds for $W=W_0$.
\end{lemma}
\begin{proof}
We need to show that $\mu(D_{W_0}\bigcap \fdom)=\frac{1}{2}$ and the values of $H(\bol{\alpha}(p))$ are uniform in $[0, 1]$ with respect to the hyperbolic metric on $D_{W_0}\bigcap \fdom$. Write $p=x+iy$ for $-\frac{1}{2}\leq x\leq 0$ and $y\geq 1$, and from Proposition \ref{prop:Hbound}, we have $H(\bol{\alpha}(p))=1-\frac{1}{y}$. Thus,
\[H(\bol{\alpha}(p))\geq 1-\epsilon \Leftrightarrow y\geq \dfrac{1}{\epsilon}.\]
The hyperbolic area of this region is
\[\int_{-1/2}^{0}\int_{1/\epsilon}^{\infty}\dfrac{1}{y^2}\,dy\, dx=\dfrac{\epsilon}{2}.\]
This grows linearly with $\epsilon$, which implies the uniform distribution. Taking $\epsilon=1$ gives the claimed hyperbolic area of the whole region.
\end{proof}

To work with $W_k$ for $k\geq 1$, we show that $C_W^{\epsilon}$ is divided into a number of equal parts, and thus we can still use Proposition \ref{prop:heightuniform}. Before doing $k\geq 2$, we need a lemma about M\"{o}bius maps and circles.

\begin{lemma}\label{lem:SUaction}
Let $C$ be a circle in $\CC$ with centre $p$ and radius $r$, which does not contain the origin. Let $C'$ be the image of the circle under the M\"{o}bius transformation $SU$, where $S$ and $U$ are as in Proposition \ref{prop:ap1orbit}. Write $|p|=d$, and let $C'$ have centre $p'$ and radius $r'$. Then
\[\arg(p')=\arg(p), \qquad |p'|=\dfrac{d}{d^2-r^2}, \qquad r'=\dfrac{r}{d^2-r^2}.\]
\end{lemma}
\begin{proof}
The action of $SU$ on a point is via
\[re^{i\theta}\rightarrow \dfrac{1}{r}e^{i\theta}.\]
In particular, if $f(C)$ is the furthest point from the origin on $C$ and $c(C)$ is the closest point to the origin on $C$, then $SU$ swaps $f(C)$ and $c(C')$, as well as $c(C)$ and $f(C')$. The centre of $C'$ is the midpoint of $c(C')$ and $f(C')$, and the radius is half of the distance between these points. A direct computation finishes the claim.
\end{proof}

\begin{lemma}\label{lem:Wk2}
Theorem \ref{thm:depthprob} holds for $W=W_k$ when $k\geq 2$.
\end{lemma}
\begin{proof}
The unit circle splits $D_{W_k}$ into two pieces: call the upper piece $R_1$, and the lower $R_2$. See Figure \ref{fig:W1decomp} for the picture when $k=2$.

\begin{figure}[htb]
	\includegraphics[scale=1]{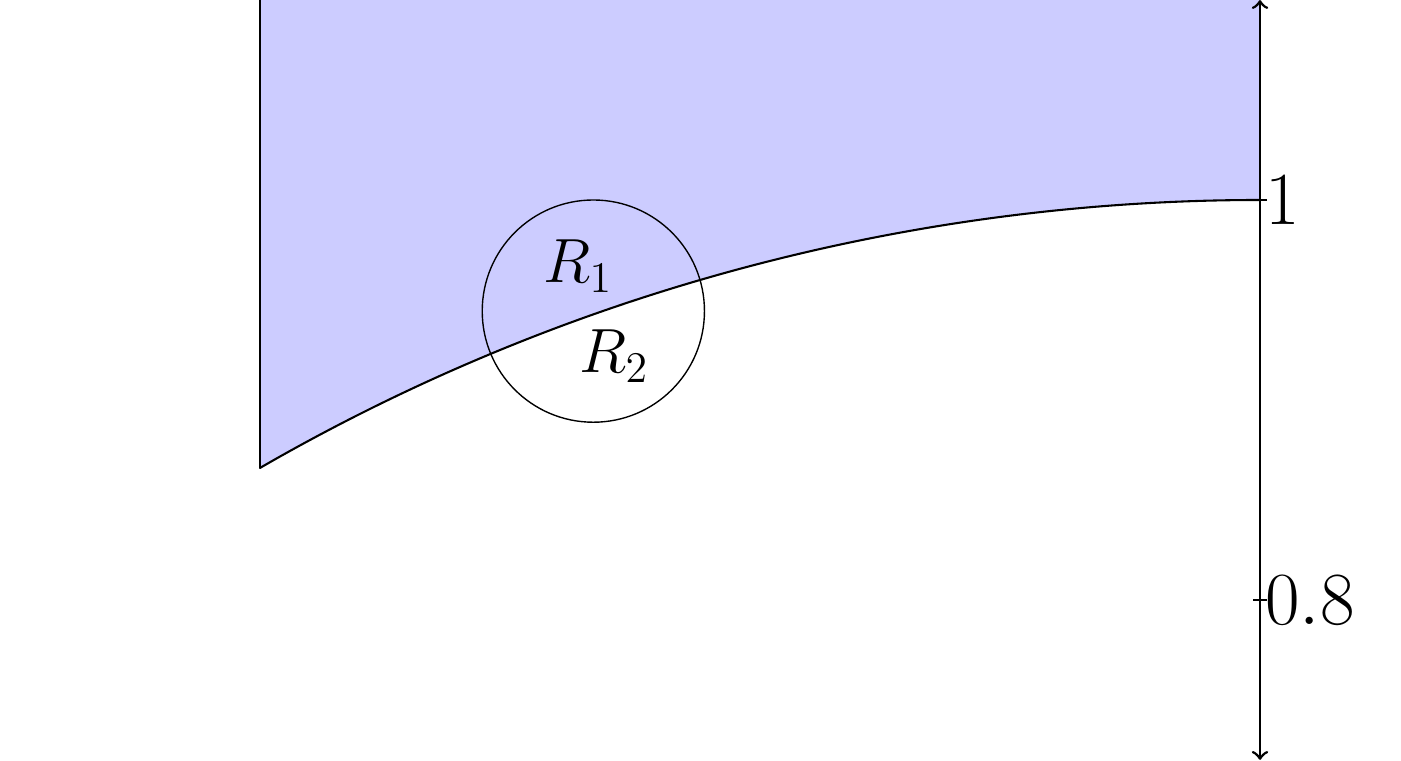}
	\caption{The two regions of $D_{W_2}$.}\label{fig:W1decomp}
\end{figure}

We claim that the M\"{o}bius map $SU$ swaps $R_1$ and $R_2$ and preserves $C_{W_k}^{\epsilon}$. If this holds, it will swap $R_1\bigcap C_{W_k}^{\epsilon}$ and $R_2\bigcap C_{W_k}^{\epsilon}$, hence the height distribution follows from Proposition \ref{prop:heightuniform}. The final hyperbolic area will be half of $\mu(D_{W_k})$, which was computed in Lemma \ref{lem:Wk3}.

Since $SU$ preserves the unit circle, sending the inside to the outside, $R_1$ and $R_2$ swap. The explicit expression for $C_{W_k}^{\epsilon}$ is given in Proposition \ref{prop:heightuniform}, and adopting the notation of Lemma \ref{lem:SUaction}, we have
\[d^2=\left(\dfrac{v_k}{w_k}\right)^2+\left(\dfrac{\sqrt{t_k^2-1}+\epsilon}{2w_k}\right)^2=\dfrac{4u_kw_k+\epsilon^2+2\epsilon\sqrt{t_k^2-1}}{4w_k^2},\]
where we used Corollary \ref{cor:tuvweqn} to simplify. The radius is given by
\[r^2=\dfrac{\epsilon^2+2\epsilon\sqrt{t_k^2-1}}{4w_k^2}.\]
Hence
\[d^2-r^2=\dfrac{u_k}{w_k}=1,\]
by the computation in Lemma \ref{lem:wkcentre}. Finally, Lemma \ref{lem:SUaction} shows that the centre and radius are unchanged, hence the circle is preserved.
\end{proof}

The last case is $W=W_1$.

\begin{lemma}\label{lem:Wk1}
Theorem \ref{thm:depthprob} holds for $W=W_1$.
\end{lemma}
\begin{proof}
Similarly to Lemma \ref{lem:Wk2}, it suffices to split $D_{W_1}$ into six pieces, and show that there are M\"{o}buis transformations that permute all six pieces while fixing $C_{W_1}^{\epsilon}$. The decomposition is provided by the unit circle ($C_1$), the circle $C_2:(x+1)^2+y^2=1$, and the line $C_3:x=-1/2$; see Figure \ref{fig:W0decomp} for the labeling of the six regions.

\begin{figure}[htb]
	\includegraphics[scale=0.7]{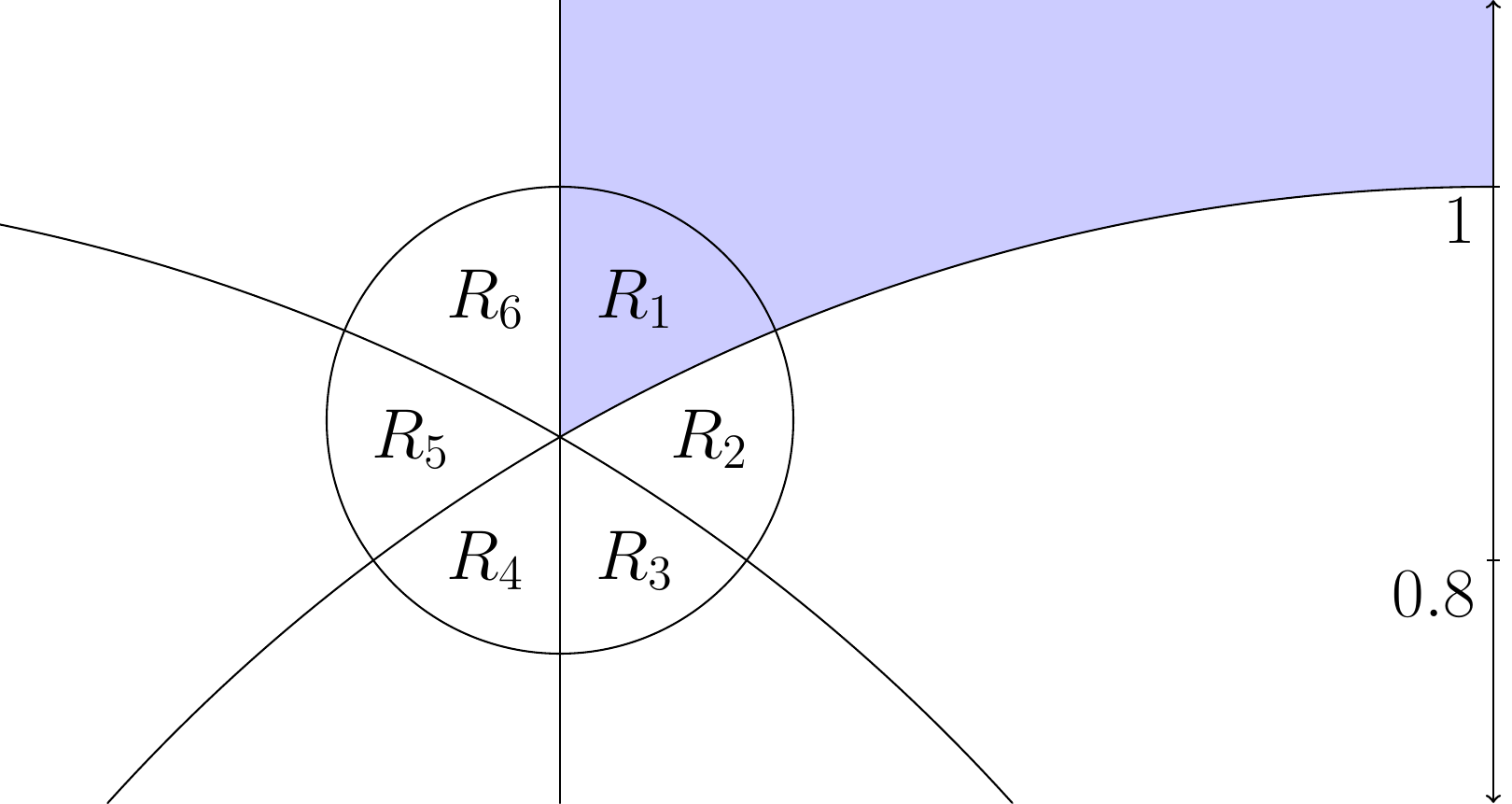}
	\caption{The six regions of $D_{W_1}$.}\label{fig:W0decomp}
\end{figure}

We are working with the coefficient quadruple $(t, u, v, w)=(7, 4, -2, 4)$, so the centre and radius of $C_{W_1}^{\epsilon}$ are given by
\[p=\dfrac{-1}{2}+\dfrac{\sqrt{48}+\epsilon}{8}i, \qquad r^2=\dfrac{\epsilon^2+2\epsilon\sqrt{48}}{8}.\]
Start with the M\"{o}bius transformation $T^{-1}U$, which corresponds to a reflection across the line $x=-1/2$. It is clear that this swaps regions $R_1$ and $R_6$, $R_2$ and $R_5$, $R_3$ and $R_4$, as well as preserving $C_{W_1}^{\epsilon}$.

Next, consider $T^{-1}S$, which sends $C_1\rightarrow C_2\rightarrow C_3\rightarrow C_1$. Furthermore, it also permutes the regions by $R_1\rightarrow R_5\rightarrow R_3\rightarrow R_1$ and $R_2\rightarrow R_6\rightarrow R_4\rightarrow R_2$. If it preserves $C_{W_1}^{\epsilon}$, we will be done, since we can combine $T^{-1}S$ and $T^{-1}U$ in an appropriate way to preserve $C_{W_1}^{\epsilon}$ and send $R_k$ to $R_1$ for all $1\leq k\leq 6$. 

Since $T^{-1}U$ fixes $C_{W_1}^{\epsilon}$, it suffices to show that $(T^{-1}S)^{-1}(T^{-1}U)=S^{-1}U=SU$ preserves $C_{W_1}^{\epsilon}$. This was done in Lemma \ref{lem:Wk2} for $C_{W_k}^{\epsilon}$ with $k\geq 2$, and the proof still works when $k=1$.
\end{proof}

Combining Lemmas \ref{lem:Wk3}, \ref{lem:Wk0}, \ref{lem:Wk2}, and \ref{lem:Wk1} completes the proof of Theorem \ref{thm:depthprob}.

\section{Integral packings and spikes}\label{sec:spikes}

To specialize our results to integral packings, let $n$ be a positive integer, and consider choosing a random $\bol{q}\in\ID(n)$, i.e. an $\Ap_1-$orbit of a primitive integral Descartes quadruple starting with curvature $n$. As shown in \cite{GLMWY02}, this set has size $h^{\pm}(-4n^2)$, the number of $\PGL(2, \ZZ)-$equivalence classes of PDBQFs with discriminant $-4n^2$. This fact can also be deduced from Proposition \ref{prop:ap1orbit}.

Take $S_n:=\{p_{\bol{q}}:\bol{q}\in \ID(n)\}$ to be the set of all principal roots of elements of $\ID(n)$, considered as a subset of the fundamental domain $\fdom$. A classic theorem of Duke (\cite{Duke88}) says that these points equidistribute as $n\rightarrow\infty$. In particular, we can apply Theorem \ref{thm:depthprob}!

\begin{theorem}\label{thm:integraldepthprob}
Let $n$ be a positive integer, let $W$ be a depth element, and take $d_W$ and $t$ as in Theorem \ref{thm:depthprob}. Then as $n\rightarrow\infty$, the probability that $W$ is a depth element for a randomly chosen element of $\ID(n)$ tends to $d_W$. Furthermore, the heights of such elements tend to a uniform distribution on $[0, t-\sqrt{t^2-1}]$.
\end{theorem}

\subsection{A precise description of the Apollonian staircase}\label{sec:staircase}

Theorem \ref{thm:integraldepthprob} immediately tells us how to describe the Apollonian staircase, as depicted in Figures \ref{fig:RMC1} and \ref{fig:RMC2}, hence proving Theorem \ref{thm:getstairs}. For each depth element $W$ which intersects $\fdom$, let $(t, u, v, w)$ be the corresponding coefficient triple. Then $W$ contributes a single ``step'' from $0$ to $t-\sqrt{t^2-1}$ with height $\frac{d_W}{t-\sqrt{t^2-1}}$, where $d_W$ is as in Theorem \ref{thm:depthprob}. As long as $W\neq \Id_2$, this is given by 
\[\frac{a_W}{\sqrt{t^2-1}}.\]
To construct the staircase, order the depth elements by $t$, and stack the stairs on top of each other, one depth element at at time.

Explicitly, the first 6 stairs (to 10 decimal places) are given in Table \ref{table:first6stairs}. Note that the last two stairs have the same value of $t$, and combine to give a ``super-stair''.
\begin{table}[hbt]
\centering
\caption{The first 6 stairs.}\label{table:first6stairs}
\begin{tabular}{|c|c|c|c|} 
\hline
$W$            & $t$ & $t-\sqrt{t^2-1}$ & Height \\ \hline
$\Id_2$        &  1 &                 1 & 0.9549296586 \\ \hline
$S_1$          &  7 &      0.0717967697 & 0.2886751346 \\ \hline
$S_4S_1$       & 17 &      0.0294372515 & 0.3535533906 \\ \hline
$S_1S_4S_1$    & 31 &      0.0161332303 & 0.1936491673 \\ \hline
$S_3S_4S_1$    & 49 &      0.0102051443 & 0.2449489743 \\ \hline
$S_4S_1S_4S_1$ & 49 &      0.0102051443 & 0.1224744871 \\ \hline
\end{tabular}
\end{table}

While a general formula for the stairs does not seem plausible, this process allows one to exactly compute any given stair. Note that contributions to the bottom stair ($W=\Id_2$) are from circles that are part of a Descartes quadruple containing the minimal curvature in the packing. In other words, they are precisely the circles that are tangent to the outermost circle. This proves Corollary \ref{cor:probtangent}.

\subsection{Spikes}

Most of the results so far apply equally to integral Descartes quadruples as non-integral quadruples. On the other hand, the occurrence of spikes, as seen in Figure \ref{fig:RMC2}, is something specific to the integral case. The heights of the spikes relative to the bottom stair height depends on bin size, and is thus a bit artificial. In particular, we will only talk about the approximate heights of the spikes, as opposed to a precise description.

\begin{definition}
Let $c_1, c_2$ be integers. The tangency number of $c_1, c_2$, denoted $T(c_1, c_2)$, is equal to the number of primitive integral $c_1-$quadruple classes that contain a quadruple with $c_2$ as a curvature. 
\end{definition}

Essentially, $T(c_1, c_2)$ is equal to the number of primitive integral Apollonian circle packings that contain circles of curvatures $c_1$ and $c_2$ that are tangent. 

\begin{definition}
Let $n$ be a positive integer, and let $\RMC_0(n)$ denote the multiset of ratios of minimal curvatures to $n$, where we only count the bottom stair of the Apollonian staircase. In other words,
\[\RMC_0(n):=\{\MC(\bol{q})/n:\bol{q}\in\ID(n)\text{ has depth element $\Id_2$}\}.\]
\end{definition}

For each integer $0\leq c<n$, the multiplicity of $c/n$ in $\RMC_0(n)$ is $T(n, -c)$. When creating the histogram for $\RMC(n)$, we group together points in small ranges, and add up the corresponding multiplicities. Spikes will occur when a certain value of $c$ has $T(n, -c)$ differing greatly from its ``expected value'', i.e. when there is a large variation in the values of $T(n, -c)$ on the given range. Smaller bin sizes will accentuate the appearance of spikes, whereas larger bin sizes will start to wash away their effect.

To study the expected value, we go back to $c_1$ and $c_2$, where we can assume that $c_1+c_2>0$. Each quadruple counted in $T(c_1, c_2)$ corresponds to a quadruple $(c_1, c_2, a, b)$, which is unique up to the action by words in $S_4,P_{(34)}$, i.e.
\[(c_1, c_2, a, b)\sim (c_1, c_2, a, 2c_1+2c_2+2a-b)\sim (c_1, c_2, b, a).\]
Using the bijection $\phi$, this corresponds to the BQF quadruple equivalence
\[[c_1, c_1+c_2, c_1+c_2+a-b, c_1+a]\sim [c_1, c_1+c_2, -c_1-c_2-a+b, c_1+a]\sim [c_1, c_1+c_2, c_1+c_2+b-a, c_1+b].\]
Write $[A, B, C]=[c_1+c_2, c_1+c_2+a-b, c_1+a]$, which is a primitive integral binary quadratic form of discriminant $-4c_1^2$. The equivalence is thus
\[[A, B, C]\sim [A, -B, C]\sim [A, 2A-B, A+C-B],\]
which gives the orbit of the group $\sm{1 & x\\ 0 &\pm 1}$ for $x\in\ZZ$. There is a unique representative for each orbit with $0\leq B\leq A$, which proves the following lemma.

\begin{lemma}\label{lem:Tc1c2basic}
Let $c_1+c_2=A>0$. Then $T(c_1, c_2)$ is equal to the number of integral solutions to $B^2-4AC=-4c_1^2$ with $\gcd(A, B, C)=1$ and $0\leq B\leq A$.
\end{lemma}

By analyzing these conditions further, we obtain the following characterization.

\begin{lemma}\label{lem:Tc1c2further}
For a prime $p^e\mid\mid A$, let $s_p=s_p(e, c_1)$ denote the number of solutions $x\pmod{p^e}$ to
\[x^2\equiv -c_1^2\pmod{p^e}\quad\text{and}\quad p\nmid\gcd\left(x, \dfrac{x^2+c_1^2}{p^e}\right).\]
Then
\[0\leq T(c_1, c_2)-\dfrac{1}{2}\prod_{p^e\mid\mid A}s_p\leq 1.\]
\end{lemma}
\begin{proof}
Adopting the notation of Lemma \ref{lem:Tc1c2basic}, $B$ is even, so write $B=2x$. The equation rearranges to
\[C=\dfrac{x^2+c_1^2}{A},\]
so we have a solution (ignoring the other two conditions) if and only if $x^2\equiv -c_1^2\pmod{A}$. Next, we claim that the condition $\gcd(A, 2x, C)=1$ can be replaced by $\gcd(A, x, C)=1$. If not, then there is a situation where $A$ and $C$ are even, but $x$ is odd. Since $2\mid A\mid x^2+c_1^2$, $c_1$ must also be odd. However $x^2+c_1^2\equiv 2\pmod{4}$, so $C=\frac{x^2+c_1^2}{A}$ must be odd (or not integral), contradiction.

Next, we claim that $\gcd(A, x, C)=1$ can be deduced from $x\pmod{A}$ only. To this end, assume there is a prime $p$ with $p\mid\gcd(A, x, C)$. Write $x=p^fu$, where $p^f=\gcd(x, p^e)$, and $f\geq 1$. We know $u\pmod{p^{e-f}}$, whence we know $x^2\pmod{p^{e+f}}$. In particular, we know $x^2+c_1^2$ modulo $p^{e+1}$, and thus $C\pmod{p}$. Therefore this condition does not depend on the representative of the equivalence class $x\pmod{A}$.

The final condition is $0\leq 2x\leq A$. If $x=x'$ is a solution to $x^2\equiv -c_1^2\pmod{A}$, then there will be exactly one solution to $0\leq 2x\leq A$ from the equivalence classes $x\equiv\pm x'\pmod{A}$. This is two distinct classes unless $x\equiv 0\pmod{A}$ or $x\equiv A/2\pmod{A}$ are solutions. In particular, dividing the number of solutions $x\pmod{A}$ to $x^2\equiv -c_1^2\pmod{A}$ and $\gcd\left(A, x, \frac{x^2+c_1^2}{A}\right)=1$ by $2$ yields $T(c_1, c_2)$, where we undercount by $0$, $\frac{1}{2}$, or $1$.

Finally, by the Chinese remainder theorem, it suffices to solve this for all prime powers dividing $A$, and multiply the number of solutions together.
\end{proof}

To understand $T(c_1, c_2)$, it suffices to understand $s_p$ for all $p^e\mid\mid A$. The generic case is when $p\nmid c_1$, where it is clear that
\[s_p=\begin{cases}
2 & \text{if $p\equiv 1\pmod{4}$;}\\
1 & \text{if $p^e=2$;}\\
0 & \text{otherwise.}
\end{cases}\]

Next, if $p\mid c_1$ and $e=1$, then $x\equiv 0\pmod{p}$. However, $p\mid\frac{x^2+c_1^2}{p}$, so the $\gcd$ condition fails and $s_p=0$. On the other hand, if $p\mid c_1$ and $e=2$, then $x\equiv px'\pmod{p^2}$ for some $0\leq x'\leq p-1$, and the $\gcd$ condition becomes $x'^2+(c_1/p)^2\not\equiv 0\pmod{p}$. This always has $p-2$, $p-1$, or $p$ solutions:
\[s_p=\begin{cases}
p & \text{if $e=2$, $p\mid\mid c_1$, $p\equiv 3\pmod{4}$;}\\
p-1 & \text{if $e=2$ and ($p^2\mid c_1$ or $p=2$);}\\
p-2 & \text{if $e=2$, $p\mid\mid c_1$, $p\equiv 1\pmod{4}$;}\\
0 & \text{if $e=1$.}
\end{cases}\]
This change in behaviour is enough to introduce variation in the histogram of $\RMC_0(n)$, where larger primes $p$ induce larger variations.

The final case is $p\mid c_1$ and $e\geq 3$. It follows that $x=px'$, with $x'$ defined modulo $p^{e-1}$, and
\[x'^2\equiv -(c_1/p)^2\pmod{p^{e-2}}\quad\text{and}\quad p\nmid\dfrac{x'^2+(c_1/p)^2}{p^{e-2}},\]
whence $x'$ is counted in $s_p(e-2,c_1/p)$. This counts solutions modulo $p^{e-2}$, so going up to $p^{e-1}$ multiplies the count by $p$. When counting $s_p(e-2,c_1/p)$, we have the slightly less restrictive condition of
\[p\nmid\gcd\left(x', \dfrac{x'^2+(c_1/p)^2}{p^{e-2}}\right).\]
If $p\mid c_1/p$, then $p\mid x'$ necessarily, whence all lifted solutions are valid. If $p\nmid c_1/p$, then we must consider the $p$ solutions modulo $p^{e-1}$, i.e. $x'+kp^{e-2}$ for $k=0,1,\ldots, p-1$. If $p\neq 2$, then exactly one of these fails to lift. If $p=2$, then $s_p(e-2, c_1/p)=0$ unless $e=3$, and it can be seen that both solutions lift. In particular, if $e\geq 3$ and $p\mid c_1$, then
\[s_p(e, c_1)=\begin{cases}
ps_p(e-2, c_1/p) & \text{if $p=2$ or $p^2\mid c_1$;}\\
(p-1)s_p(e-2, c_1/p) & \text{if $p$ is odd and $p\mid\mid c_1$.}
\end{cases}\]

By inducting and considering the various cases, it can be shown that

\begin{lemma}\label{lem:spvalues}
Assume that $p^f\mid\mid c_1$. If $p$ is odd, then
\[s_p(e, c_1)=\begin{cases}
2(p-\mathds{1}_{f>0})p^{f-1} & \text{if $e\geq 2f+1$ and $p\equiv 1\pmod{4}$;}\\
0 & \text{if $e\geq 2f+1$ and $p\equiv 3\pmod{4}$;}\\
(p-2)p^{f-1} & \text{if $e=2f$ and $p\equiv 1\pmod{4}$;}\\
p^f & \text{if $e=2f$ and $p\equiv 3\pmod{4}$;}\\
(p-1)p^{e/2-1} & \text{if $e<2f$ and $e$ is even;}\\
0 & \text{if $e<2f$ and $e$ is odd.}
\end{cases}\]
If $p=2$, then
\[s_2(e, c_1)=\begin{cases}
0 & \text{if $e\geq 2f+2$;}\\
2^f & \text{if $e=2f+1$;}\\
2^{e/2-1} & \text{if $e\leq2f$ and $e$ is even;}\\
0 & \text{if $e\leq 2f$ and $e$ is odd.}
\end{cases}\]
\end{lemma}

The main takeaway is that for $s_p(e, c_1)$ to be larger than normal, we must have $e\geq 2$ and $f\geq 1$. Furthermore, the size of $s_p(e, c_1)$ is approximately $p^{\min(e/2, f)}$ when it is non-zero. Specializing back to $\RMC(n)$, we obtain the following theorem.

\begin{theorem}\label{thm:finalspikes}
Let $n$ be a positive integer. Spikes occur in the histogram of $\RMC_0(n)$ (and $\RMC(n)$) for each prime $p\leq\sqrt{n}$ that divides $n$. These spikes occur near $\frac{c}{n}$ where $0\leq c\leq n-1$ is an integer with $p^2\mid n-c$, and larger values of $p$ give larger spikes.
\end{theorem}
\begin{proof}
A spike occurs near $\frac{c}{n}$ when $T(n, -c)$ is larger than normal. Lemma \ref{lem:Tc1c2further} implies that this happens when a value of $s_p$ is large, for some $p\mid n-c$. Lemma \ref{lem:spvalues} implies that $p\mid n$ and $p^2\mid n-c$, hence $n\geq c+p^2\geq p^2$. The lemma also implies that larger $p$ gives larger spikes, since the value of $s_p$ grows as a power of $p$.
\end{proof}

If $n$ is prime, then it has no prime divisors at most $\sqrt{n}$, so by Theorem \ref{thm:finalspikes}, there are no spikes! An example of this was already seen in Figure \ref{fig:RMC1}, where the histogram was very smooth.

The effect of small prime powers is low for two reasons: they create the least variation, and the bin size required to make a good histogram ends up grouping enough terms together. In turn, this creates more of a fuzzy effect, as opposed to isolated spikes. See Figure \ref{fig:RMC3} for the example of $n=2\cdot 3\cdot 5\cdot 1110023$.

\begin{figure}[htb]
	\includegraphics{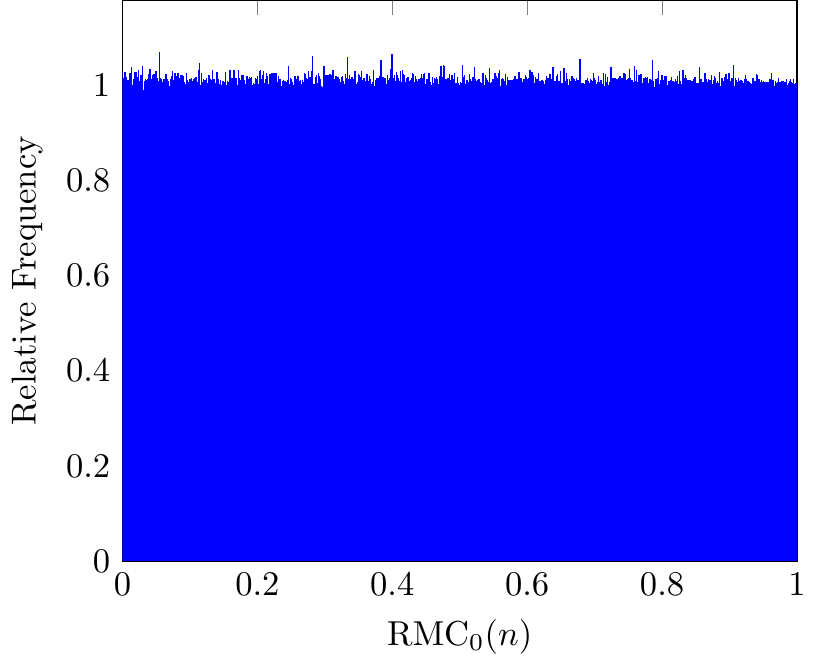}
	\caption{Histogram for $n=33300690$; $8479975$ data points in $2000$ bins.}\label{fig:RMC3}
\end{figure}

\begin{figure}[htb]
	\includegraphics{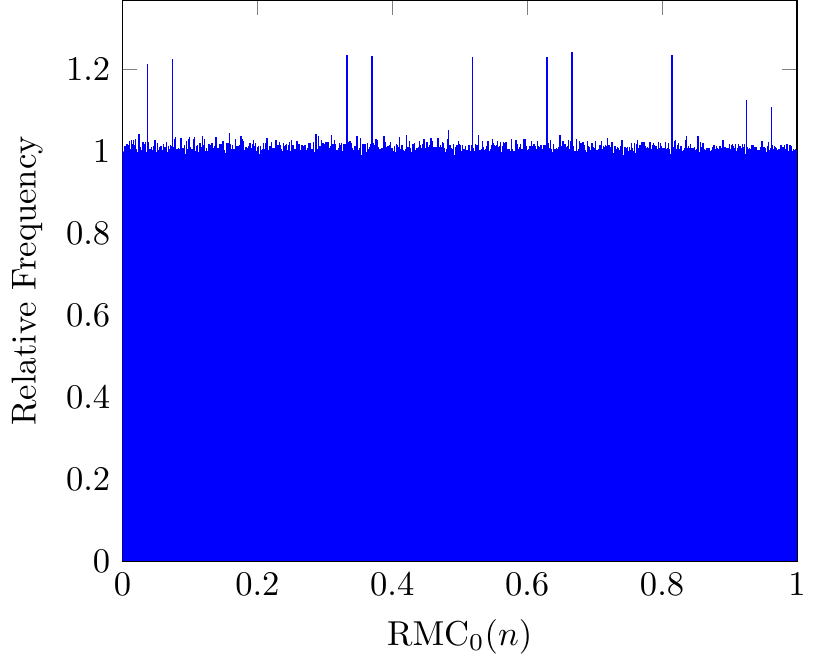}
	\caption{Histogram for $n=26516187$; $8449251$ data points in $2000$ bins.}\label{fig:RMC4}
\end{figure}

Finally, consider $n=26516187=3^3\cdot 991^2$, as depicted in Figure \ref{fig:RMC4}. Large spikes occur near $1-\frac{c_i}{27}$, corresponding to
\[991^2\mid n-(27-c_i)991^2=c_i991^2.\]
However, we note that these spikes do not occur for each value of $c_i$, namely they occur when
\[c_i\in\{1, 2, 5, 9, 10, 13, 17, 18, 25, 26\}.\]
Furthermore, when $c_i\in\{1,2\}$, the spikes are about half the size! This is explained fully by Lemmas \ref{lem:Tc1c2further} and \ref{lem:spvalues}: the $991^2$ causes $s_{991}$ to be abnormally large, but prime powers that divide $c_i$ also contribute to the extra height. When $c_i\in\{1,2\}$, this is an extra factor of $1$, whereas the other $c_i$ give an extra factor of $2$, explaining the height difference. The values of $c_i$ not listed all had a prime factor with $s_p=0$, which completely nullified the corresponding spike.

\bibliographystyle{alpha}
\bibliography{../references}
\end{document}